\documentclass[12pt]{article}

\usepackage{amsmath,amssymb,amsthm,enumerate,graphics,epsfig,cite}%,backref
\usepackage[cp1251]{inputenc}
\usepackage[english]{babel}
\usepackage[unicode]{hyperref}

\newtheorem{theorem}{Theorem}
\newtheorem{proposition}{Proposition}
\newtheorem{corollary}{Corollary}
\newtheorem{lemma}{Lemma}

\newtheorem{remark}{Remark}

\usepackage{cite}
\usepackage[usenames]{color}
\usepackage{colortbl}
\usepackage[normalem]{ulem}

\usepackage{graphicx}
\usepackage{ifpdf}
\ifpdf
\graphicspath{{pictures/}}
  \DeclareGraphicsExtensions{.pdf,.png,.jpg}
\else
  \DeclareGraphicsExtensions{.eps}
\fi

\usepackage{caption}
\DeclareCaptionLabelSeparator{dot}{. }
\makeatletter
\renewcommand{\@seccntformat}[1]{\csname the#1\endcsname.\;}
\makeatother

\makeatletter
\def\@biblabel#1{#1. }
\makeatother

\newcommand{\AbstractE}[1]{\vspace{6mm}\par\noindent\hspace*{10mm}
\parbox{140mm}{\small { } #1} \par}

\usepackage{color}
\usepackage{listings}

\definecolor{codegreen}{rgb}{0,0.6,0}
\definecolor{codegray}{rgb}{0.5,0.5,0.5}
\definecolor{codepurple}{rgb}{0.58,0,0.82}
\definecolor{backcolour}{rgb}{0.95,0.95,0.92}
\definecolor{darkgreen}{rgb}{0.0, 0.5, 0.0}

\lstdefinestyle{mystyle}{
    backgroundcolor=\color{backcolour},
    commentstyle=\color{codegreen},
    keywordstyle=\color{magenta},
    numberstyle=\tiny\color{codegray},
    stringstyle=\color{codepurple},
    basicstyle=\footnotesize\ttfamily,
    breakatwhitespace=false,
    breaklines=true,
    captionpos=b,
    keepspaces=true,
    numbers=left,
    numbersep=5pt,
    showspaces=false,
    showstringspaces=false,
    showtabs=false,
    tabsize=2,
    language=Python
}

\usepackage{algorithm}
\usepackage{algpseudocode}

\begin{document}
\renewcommand\refname{\large References}
\textbf{Oksana Satur} (Institute of Mathematics of the National Academy of Sciences of Ukraine, 3 Tereshchenkivska Str., 01024 Kyiv, Ukraine)

\begin{center}
\Large\bf
Mean--field dynamics of attractive resource interaction: From uniform to aggregated states
\end{center}

\AbstractE{We introduce and study a nonlinear discrete dynamical system describing the evolution of a resource distribution among interacting agents.
The model generalizes several classical mean--field and opinion--dynamics frameworks and is defined on the standard simplex, where each coordinate evolves according to an interaction rule depending on preference--based mean--field interactions.
We provide a complete analytical description of the long-term behavior of the system.
First, we establish monotonicity properties and show that the dynamics always remains in a positively invariant region determined by initial conditions. We prove the existence of a unique fixed point for any admissible parameter set and derive an explicit closed-form formula for the equilibrium in arbitrary dimension.
We then analyze the local stability of the fixed point and identify parameter regimes leading to aggregation or uniform distributions. Finally, we characterize all possible asymptotic scenarios and show that, despite the nonlinear structure, the system does not exhibit oscillatory or chaotic behavior: every trajectory converges to the unique equilibrium.
The results provide a full qualitative theory for this class of monotone resource--interaction models and offer a mathematical explanation for the transition from uniform to aggregated states.
\bigskip


\bigskip

}

\bigskip

\section{Introduction}
\label{section34}
{\hspace{0.5cm}}
Modeling systems composed of many interacting components that compete for a limited resource is a fundamental task in numerous scientific fields \cite{Sayama2015}. These models are central to sociophysics, where methods from statistical physics are applied to understand how macroscopic social patterns, such as consensus or fragmentation, emerge from microscopic interactions \cite{Starnini2025, 18}.

Such models are applied in the description of social and economic systems, in network theory, and in the analysis of market share or political influence distribution. A key question in the study of such systems is predicting their long-term behavior: will the system converge to a state of equilibrium where the resource is distributed uniformly among all participants, or will it, on the contrary, lead to an aggregated state where the resource is concentrated in only a few components?

The type of interaction between components~-- repulsive or attractive~-- critically affects the system's dynamics. Models with repulsive interaction, which often describe competition, typically lead to the system's tendency to avoid resource concentration. In contrast, models with attractive interaction, which are the focus of this paper, can describe processes of cooperation, gravitation towards centers of influence, or capital concentration. It is such systems that are the subject of our investigation. While the term ``conflict'' might seem counterintuitive in a model with attractive interaction, it represents a competition for dominance or influence within the state space~$\Omega$. In this context, attraction towards more influential components is a key strategic element of the conflict, leading to the aggregation of resources rather than their uniform dispersion.
This connects to the foundational concept of ``conflict'' between probability measures introduced by Volodymyr Koshmanenko \cite{KoTC1, KoTC}. While the original work focused on quantifying divergence, the attractive dynamics explored here demonstrate how such interactions can lead to ``conflict resolution'' through convergence towards a single, stable equilibrium. The mathematical modeling of interaction between conflicting systems, especially concerning the role of external support, is a subject of ongoing research \cite{Karataieva2024}.

This work can be framed within the rich tradition of resource competition models, prevalent in mathematical ecology~\cite{SchF68} and economics~\cite{Levin1970}. Furthermore, studies of competition for substitutable resources, where bifurcation analysis reveals transitions between exclusion and coexistence, provide a valuable context for the bifurcation analysis we undertake in Section \ref{subsection342} \cite{DellalBL2022}.

An alternative and powerful interpretation is as a system of interacting agents striving to reach consensus. In this view, the state vector represents a distribution of opinions, making the model an example of a discrete--time multi--agent system \cite{6, 8}. Aggregated states then represent opinion polarization or fragmentation, phenomena widely studied in sociophysics \cite{10, Sobkowicz2016, Starnini2025}.

Finally, the model shares properties with econophysics models of wealth distribution. The conservation law $\sum_i p_i^t = 1$ mirrors a closed economy. While stochastic agent-based models like that of Chakraborti \& Chakrabarti (2000) lead to Boltzmann--Gibbs distributions \cite{Chakraborti2000}, our deterministic model provides an alternative mechanism for wealth aggregation, where the parameters $c_i$ can be interpreted as the ``economic attractiveness'' of an agent, analytically determining the final distribution.

\looseness=1 Next, we construct a model of a dynamical system that describes the uniform or aggregated distribution of a certain resource in a given state space $\Omega= \{\omega_1, \omega_2, \ldots, \omega_n \}$, $n>1$.
We fix a certain discrete probability measure $\mu$, which corresponds to a specific distribution. The value of this measure,
$$
\mu(\omega_i) =: p_i \geq 0, \quad i = 1, \ldots, n,
$$
is interpreted as the share of the resource or the relative weight of the i-th component of the system.
By construction, $\sum_{i} p_{i} = 1$. Therefore, the sequence $\mu(\omega_i)$ defines a stochastic vector in
${\mathbb R}_+^{n}$: ${\bf p}=(p_1, p_2, \ldots, p_n)$.

The competition for resources concentrated in each of the states $\omega_i$ is described in terms of the measure~$\nu$, whose values at the points~$\omega_i$  are defined as follows:
 $$\nu(\omega_i) = \frac{1-\mu(\omega_i)}{n-1}=\frac{1-p_i}{n-1}=: r_i.$$
It is easy to verify that these values also form a stochastic vector ${\bf r}=(r_1, r_2, \ldots, r_n) \in {\mathbb R_+^{n}}$.
The quantities $\nu(\omega_i)$ can be interpreted as an average distribution (mean field) of $\mu$ over the subset $\Omega \setminus \omega_i$.
This mean--field approach is widely used in statistical physics to describe systems where individual components interact with an average property of the entire system rather than with each other individually \cite{CastellanoFL2009}.
Furthermore, this paradigm has been substantially broadened by the theory of mean--field games, which provides a rigorous framework for analyzing strategic interactions within large populations of agents. Modern applications of this theory are highly diverse, ranging from ranking games with diffusion control \cite{Ankirchner2024} to spatial optimization on bounded domains \cite{Alharbi2026}. While our model does not explicitly formulate a stochastic game, it shares the core mean--field philosophy, describing how the aggregate state of the system shapes the favorable conditions for individual components.

Conceptually, the interdependence of components through the global normalization constraint ($\sum_i p_i = 1$) and the shared mean--field distribution establishes a profound link with the notion of correlated equilibrium. In classical game theory, a correlated equilibrium arises when players coordinate their strategies via a shared signal~\cite{Aumann1974}. In the context of large populations and
mean--field games, recent literature introduces mean--field analogues of correlated equilibrium, where the population distribution plays the role of a shared correlating signal~\cite{HartSchmeidler1989, CampiF2022}. In our macroscopic system, the mean--field vector~${\bf r}^t$ and the normalizing denominator $z^t$ fulfill this exact function, naturally driving the system toward aggregated, cooperative outcomes rather than isolated competitive Nash equilibria.

We investigate the behavior of the trajectories of a discrete-time dynamical system in terms of stochastic vectors
\begin{equation}\label{eq:1}
 {\bf p}^t \stackrel{\divideontimes, t }{\longrightarrow } {\bf p}^{t+1}, \quad t=0,1,\ldots \quad ({\bf p}^0\equiv {\bf p}),
\end{equation}
where $\divideontimes$ denotes the transformation governing the system's dynamics over the state space~$\Omega$.

In coordinate terms, the law of conflict dynamics is defined as:
\begin{equation}
\label{eq:3.1}
p^{t+1}_i=\frac {p^{t}_i\big(1+{c_{i}}r^{t}_i\big)}{z^{t}},  \quad 0 < c_i \leq 1, \quad i = 1, \ldots, n,
\end{equation}
where $r^{t}_i$ is the mean--field term defined as
\begin{equation}
\label{eq:3.1r}
r^{t}_i=\frac{1-p^{t}_i}{n-1}.
\end{equation}
The parameters $c_i$ represent the time-invariant ($t$-independent) favorable conditions explicitly associated with each state~$\omega_i$, thereby dictating the strength of the attractive interaction for that component.
Finally, the normalizing denominator $z^t$, defined explicitly as
\begin{equation*}
z^t = \sum_{j=1}^n p^{t}_j\big(1+{c_{j}}r^{t}_j\big),
\end{equation*}
ensures the stochasticity of the vector ${\bf p}^{t+1}$ (i.e., $\sum_{i=1}^n p_i^{t+1} = 1$).

The original mathematical idea underlying the transformation~$\divideontimes$ stems from the ``theorem of conflicts'' for pairs of discrete probability measures (or stochastic vectors), introduced by V.~Koshmanenko \cite{KoTC1, KoTC}. The foundational law of conflict dynamics was defined in its generic symmetric form as:
\begin{equation*}\label{simeq}
  p_i^{t+1}= \frac{1}{z^t}p_i^t(1- r_i^t), \quad r_i^{t+1}= \frac{1}{z^t}r_i^t(1- p_i^t),
\end{equation*}
where $z^t = 1 - ({\bf p}^t, {\bf r}^t)$ is the normalizing denominator (with $(\cdot, \cdot)$ denoting the scalar product in $\mathbb{R}^n$). The governing equation~\eqref{eq:3.1} studied in this work is a direct adaptation of this original equation, obtained by explicitly defining the opposing vector ${\bf r}^t$ as the mean--field distribution~\eqref{eq:3.1r} and introducing the coefficients $c_i$ to scale the interaction strength.

It is worth noting that in works \cite{KKKK18, KK19}, a mathematical model of a conflict dynamical system describing the behavior of individuals in a society has been developed. The dynamics in that model is described by equations with repulsive interaction:
$$
p^{t+1}_i=\frac {p^{t}_i\big(1-{c_{i}}r^{t}_i\big)}{z^{t}}, \quad 0 \leq c_i \leq 1.
$$
It can be noted that the ``game for power'' model proposed in \cite{KKKK18} is a prime example of a system with repulsion, where agents seek to minimize the influence of others. The model presented in this manuscript, featuring attractive interaction (\eqref{eq:3.1}), can thus be interpreted as a ``game for cooperation'' or alliance formation, where an agent's benefit increases by gravitating towards resources controlled by others.
The choice of attractive interaction fundamentally changes the system's dynamics and significantly affects its properties.

We further analyze the behavior of the trajectories \eqref{eq:1}, defined by equations~\eqref{eq:3.1}, in terms of the coordinates of the stochastic vectors ${\bf p}^t$, ${\bf r}^t$.

\section{Uniformly distributed equilibria}\label{subsection341}
Assume that in formulas \eqref{eq:3.1}, all $c_i=1$ for all $i = 1, \ldots, n$, meaning favorable conditions are identical for each state~$\omega_i$ and remain unchanged over time. Then equation \eqref{eq:3.1} can be rewritten as
\begin{equation}
\label{eq:2}
p^{t+1}_i=\frac {p^{t}_i\big(1+r^{t}_i\big)}{1+ \left({\bf p}^{t}, {\bf r}^{t}\right)}, \quad i = 1, \ldots, n.
\end{equation}
Next, we investigate the properties of trajectories \eqref{eq:1} in this case. Note that the behavior of the dynamic system is analogous to the case when all constants $c_i$ are equal to some positive number $c$. For simplicity of proof, the case where $c_i=1$ is considered. Below, it will be proven that the change in coordinates of vectors ${\bf p}^t$, ${\bf r}^t$ according to equations \eqref{eq:2} leads to all non-zero coordinates of these vectors becoming equal to each other as $t \rightarrow \infty$, which corresponds to a uniform distribution on the space $\Omega$.

\begin{remark}
The equality ${\bf p} = {\bf r}$ is possible if and only if $p_{i}=r_{i}= \frac{1}{n}$ for all $i=1, \ldots, n$.
\end{remark}

The fixed point where $p_{i}^{\infty} = \frac{1}{n}$ is known in the multi-agent systems literature as the {consensus state} \cite{8, Li2011, OlfatiSaber2004}. It is a direct analogue to the outcome of classic linear consensus protocols, such as the DeGroot model, where iterative averaging on a fully connected graph leads to a system-wide average \cite{Hegselmann2002}.
The model presented here, in this special case of uniform $c_i$, can thus be viewed as a novel, nonlinear consensus protocol. This perspective closely aligns with recent research, such as the work \cite{SatKhar20}, which constructs a discrete-time opinion formation model of a similar type. In their framework, analogous to our uniform case, the iterative interaction mechanism effectively neutralizes competitive differences and drives the multi-agent system toward a state of absolute consensus.

If we introduce the quantity
$$L^t: = \sum\limits_{i=1}^{n} \left(p^{t}_{i}\right)^{2} = \|{\bf p}^{t}\|^{2},$$
then the transformation $\divideontimes$ can be rewritten solely in terms of the coordinates of the vector~${\bf p}^t$:
\begin{equation}
\label{eq:3}
p^{t+1}_i=p^{t}_i k^t_i, \quad k^t_i= \frac {n-p^{t}_i}{n-L^{t}}.
\end{equation}

Let $M$ denote the set of all zero coordinates of the vector ${\bf p}$, i.e.,
$$ M = \big \lbrace p_{i}\colon p_{i} = 0 \big \rbrace,$$
and let $\gamma(M)$ denote the cardinality of this set. From equations \eqref{eq:3}, it is clear that $M$ and $\gamma(M)$ do not depend on $t$.

\begin{theorem}\label{th1pop}
Every trajectory of the dynamic conflict system \eqref{eq:1} with an arbitrary pair of initial stochastic vectors ${{\bf p}, {\bf r}\in{\mathbb R}_+^n}$ $(n>1)$, whose evolution is given by formulas \eqref{eq:3}, converges to a fixed point
$$ {\bf p}^{\infty} = \lim_{t\to \infty} {\bf p}^t, \quad {\bf r}^{\infty} = \lim_{t\to \infty} {\bf r}^t.$$
Furthermore,
\begin{enumerate}
\item[$1)$] if $ \gamma(M) = 0$, then ${{\bf p}^{\infty}={\bf r}^{\infty}}$ and
$$ p_{i}^{\infty} = r_{i}^{\infty} = \frac{1}{n};$$
\item[$2)$] if $ \gamma(M) = m$, $m < n$, then ${{\bf p}^{\infty} \neq {\bf r}^{\infty}}$ and for all $p_{j} \in M $
$$ p_{j}^{\infty} = 0, \qquad r_{j}^{\infty} = \frac{1}{n-1},$$
and for all $ p_{i} \notin M $
$$p_{i}^{\infty} = \frac{1}{n-m}, \quad r_{i}^{\infty} = \frac{n-m-1}{(n-1)(n-m)}.$$
\end{enumerate}
The equilibrium is stable only in the case ${\bf p}^{\infty} ={\bf r}^{\infty} = \left( \frac{1}{n}, \frac{1}{n}, \ldots, \frac{1}{n}\right)$.
\end{theorem}
For the proof of this theorem, we will use the following statements and lemmas.

From equations \eqref{eq:3} it follows that when $p_{i}^{t}>L^{t}$
$$p_{i}^{t} > p_{i}^{t+1}, $$
and when $p_{i}^{t}<L^{t}$
$$p_{i}^{t} < p_{i}^{t+1}.$$

\begin{proposition}\label{pr1}
If $p_i \geq p_j$ for some $i$, $j$, then for all ${t = 0, 1, \ldots}$
$$ p^{t}_i \geq p^{t}_j .$$
\end{proposition}

\begin{proof}
Let $p_i \geq p_j$. Consider the difference
\begin{gather*}
\begin{split}
d^{1}_{ij} & = {p^{1}_{i}-p^{1}_{j}} = {\frac{1}{n-L}\big(p_{i}(n-p_{i})- p_{j}(n-p_{j})\big)} \\
& ={\frac{1}{n-L}\big(n(p_i-p_j)-(p_i+p_j)(p_i-p_j)\big)} ={\frac{n-p_i-p_j}{n-L}}(p_i-p_j) \geq 0.
\end{split}
\end{gather*}
Thus, ${p^{1}_{i}}\geq {p^{1}_{j}}$.
By analogous reasoning, we obtain $ {p^{t}_{i}}\geq {p^{t}_{j}}$ for any $t$.
\end{proof}

\begin{proposition}\label{pr2}
If for some $i=1, \ldots, n$
$$p_i = 0,$$ then for any $t$ $$ p^{t}_i = 0, \quad r^{t}_i = \frac{1}{n-1}.$$
\end{proposition}

\begin{proof}
This follows from equations \eqref{eq:2}.
\end{proof}

\begin{proposition} \label{pr4}
Let $p_{\min}^{t}$ and $p_{\max}^{t}$ be the minimum and maximum coordinates of the vector ${\bf p}^t$ respectively, and $L^t = \|{\bf p}^{t}\|^{2}$. Then for any ${t = 0, 1, \ldots}$ the following inequalities hold:
\begin{equation}\label{kn}
p_{\min}^{t} \leq L^{t} \leq p_{\max}^{t}.
\end{equation}
\end{proposition}

\begin{proof}
From Proposition \ref{pr1} and the stochasticity of the vector ${\bf p}^{t}$, it follows that for any $t = 0, 1, \ldots$
	\begin{equation*} p^{t}_{\min} \leq \sum_{i=1}^{n} p^{t}_{i} p^{t}_{\min} \leq \sum_{i=1}^{n} \big(p^{t}_{i}\big)^{2} \leq \sum_{i=1}^{n} p^{t}_{i} p^{t}_{\max} \leq p^{t}_{\max}. \tag*{\qed}\end{equation*}\renewcommand{\qed}{}
\end{proof}

It is obvious that the sequence ${\lbrace p^{t}_{\max} \rbrace}_{t = 0}^{\infty}$ is decreasing, since $p_{\max}^t \geq L^{t}$ for any $t$.

\begin{lemma}\label{l1}
    The sequence $L^{t} = \|{\bf p}^{t}\|^{2}$, $t = 0, 1, \ldots$ is decreasing.
\end{lemma}

\begin{proof} Let us express the difference $L^{t+1} - L^t$ using equations \eqref{eq:3}:
\begin{equation*}
L^{t+1} - L^t = \sum_{i=1}^n \left( p^t_i \cdot \frac{n-p^t_i}{n-L^t} \right)^2 - L^t = \frac{1}{\big(n-L^t\big)^2} \sum_{i=1}^n \big(p^t_i\big)^2 \left( \big(n-p^t_i\big)^2 - \big(n-L^t\big)^2 \right).
\end{equation*}

Through algebraic rearrangement, $(n-p^t_i)^2 - (n-L^t)^2 = (L^t-p^t_i)(2n-L^t-p^t_i)$,
and rewriting the second factor as $2(n-p^t_i) - (L^t-p^t_i)$, we expand the numerator:
\begin{equation*}
(n-L^t)^2 (L^{t+1} - L^t) = 2 \sum_{i=1}^n p^t_i(L^t-p^t_i) p^t_i(n-p^t_i) - \sum_{i=1}^n \big(p^t_i\big)^2\big(L^t-p^t_i\big)^2.
\end{equation*}

The second sum $\sum\limits_{i=1}^n \big(p^t_i\big)^2\big(L^t-p^t_i\big)^2 \ge 0$, so its contribution to the difference is non-positive.

Considering the first sum, we define $q^t_i = p^t_i(L^t-p^t_i)$ and $f(p_i^t) = p_i^t(n-p_i^t)$.
Notice that the sum of the elements $q^t_i$ is exactly zero:
$$\sum_{i=1}^n q^t_i = L^t \sum_{i=1}^n p^t_i - \sum_{i=1}^n \big(p^t_i\big)^2 = L^t \cdot 1 - L^t = 0.$$
Since $n \ge 2$ and $p_i^t \in [0, 1]$, the derivative $f'(p_i^t) = n-2p_i^t \ge 0$, meaning $f(p_i^t)$ is strictly increasing on $[0, 1]$.

Let us evaluate the product $q^t_i f(p^t_i)$ relative to the constant $f(L^t)$.
Since $f$ is increasing on $[0,1]$, we consider two possible cases.

\medskip
\noindent
\textit{Case 1:} $p^t_i \le L^t$.
In this case $q^t_i = p^t_i(L^t-p^t_i) \ge 0$ and, by the monotonicity of $f$,
$$f(p^t_i) \le f(L^t).$$
Therefore,
$$q^t_i f(p^t_i) \le q^t_i f(L^t).$$

\medskip
\noindent
\textit{Case 2:} $p^t_i > L^t$.
Then $q^t_i < 0$ and $f(p^t_i) > f(L^t)$. Multiplying this inequality by the
negative factor $q^t_i$ reverses its direction, which yields
$$q^t_i f(p^t_i) < q^t_i f(L^t).$$

\medskip
Thus, in both cases the inequality
$$q^t_i f(p^t_i) \le q^t_i f(L^t)$$
holds for all $i=1,\ldots,n$. Summing over $i$ gives
$$\sum_{i=1}^n q^t_i f(p^t_i) \le \sum_{i=1}^n q^t_i f(L^t) = f(L^t)\sum_{i=1}^n q^t_i = f(L^t)\cdot 0 = 0.$$

Since both terms in the expansion of $L^{t+1} - L^t$ are non-positive, we conclude that
$$L^{t+1} - L^t \le 0.$$
Thus, the sequence $\{L^t\}_{t=0}^\infty$ is decreasing.
\end{proof}

\begin{lemma}\label{l3}
If $p_i > L$, then the sequence $ {\lbrace p^{t}_i \rbrace}_{t = 0}^{\infty}$ is decreasing.
\end{lemma}
\begin{proof} Let $p_i > L$. Then
\begin{equation*}
 p_{i}^{1} - p_{i} = \frac{p_{i}(n-p_{i})-p_{i}(n-L)}{n-L} = \frac{p_{i}(L-p_{i})}{n-L} < 0,
\end{equation*}
which means $p^{1}_{i}<p_{i}$. Let's show that $p_{i}^{1}> L^{1}$. From \eqref{eq:3} we have
\begin{gather*} \begin{split}
 p_{i}^{1}-L &= \frac{np_{i} - (p_{i})^{2}-nL + (L)^{2}}{n-L} = \frac{n(p_{i}-L)-\big((p_{i})^{2}- (L)^{2}\big)}{n-L} \\
 &= (p_{i}-L)\frac{n-p_{i}-L}{n-L} >0.
\end{split}\end{gather*}
With consideration of Lemma \ref{l1}, we have $p_{i}^{1} > L^{1}$.

By the method of mathematical induction, we obtain the validity of the inequality $p_{i}^{t} > L^{t}$ for any $t$. It follows that the sequence
$\lbrace p^{t}_i \rbrace_{t = 0}^{\infty}$ is decreasing, if $p_i > L$.
\end{proof}

\begin{lemma} \label{pr3(1)}
 The sequence $L^{t} = \|{\bf p}^{t}\|^{2}$, $t = 0, 1, \ldots$ changes within the limits:
 $$\frac{1}{n-m} \leq L^{t} \leq 1, $$
 where  $m=\gamma(M)$.
\end{lemma}

\begin{proof} From the stochasticity of the vector ${\bf p}^t$ it follows that $L^{t} \leq 1$. Only in the case when $ \gamma(M) = n-1$, we have $L^{t} = 1$.

To prove the second part of the inequality, assume that $ \gamma(M) = 0$. Next, we will use the method of Lagrange multipliers to find the extrema of the function of $n$ variables $$ L\big(p_{1}, \ldots, p_{n}\big) = \sum_{i=1}^{n} \big(p_{i}\big)^{2}, \quad p_i \neq 0,$$
considering the constraint equation $$\phi\big(p_{1}, \ldots, p_{n}\big) = - 1 + \sum_{i=1}^{n} p_{i} = 0.$$

We form the Lagrangian function and denote it by $\mathcal{L}$:
$$\mathcal{L}(p_{1}, \ldots, p_{n}, \lambda) = \sum_{i=1}^{n} (p_{i})^{2} + \lambda \left(- 1 + \sum_{i=1}^{n} p_{i}\right),$$
where $\lambda$ is the so-called Lagrange multiplier.

Critical points are found from the following system:
\begin{equation*}
\begin{cases}
\frac{\partial\mathcal{L}}{\partial p_{1}} = 2p_{1}+ \lambda = 0,\\
\frac{\partial\mathcal{L}}{\partial p_{2}} = 2p_{2}+ \lambda = 0,\\
\cdots \\
\frac{\partial\mathcal{L}}{\partial p_{n}} = 2p_{n}+ \lambda = 0,\\
\phi(p_{1}, \ldots ,p_{n}) = - 1 + \sum\limits_{i=1}^{n} p_{i} = 0,
\end{cases}
\Longleftrightarrow
\begin{cases}
p_{1} = \frac{1}{n},\\
p_{2} = \frac{1}{n},\\
\cdots \\
p_{n} = \frac{1}{n},\\
\lambda = - \frac{2}{n}.
\end{cases}
\end{equation*}

Thus, we have a fixed point ${M'}\big(\frac{1}{n}, \frac{1}{n}, \ldots , \frac{1}{n}, -\frac{2}{n} \big)$.

Let's check the sufficient conditions for the extremum of the Lagrangian function. For this, we find the second partial derivatives of the function $\mathcal{L}(p_{1}, \ldots, p_{n}, \lambda)$:
$$\mathcal{L}^{''}_{p_{i}p_{i}} = 2, \quad \mathcal{L}^{''}_{p_{i}p_{k}} = 0, \quad i = 1, \ldots, n, \quad k = 1, \ldots, n, \quad i \neq k.$$

We write the second differential of the Lagrangian function:
$${\rm d}^{2}\mathcal{L}\big(p_{1}, \ldots, p_{n}, \lambda\big) = 2{\rm d}{p_{1}^{2}}+2{\rm d}{p_{2}^{2}}+ \cdots + 2{\rm d}{p_{n}^{2}} = 2\big({\rm d}{p_{1}^{2}}+{\rm d}{p_{2}^{2}}+ \cdots + {\rm d}{p_{n}^{2}}\big).$$

Since ${\rm d}^{2}\mathcal{L} > 0$, the function $\mathcal{L}(p_{1}, \ldots, p_{n}, \lambda)$ has a local minimum at point ${M'}$, $\mathcal{L}_{\min} = \mathcal{L}({M'}) = \frac{1}{n}$.

Thus, the function $L(p_{1}, \ldots, p_{n})$ under the constraint $\sum\limits_{i=1}^{n} p_{i} = 1$ has a local conditional minimum at the point $M\left(\frac{1}{n}, \frac{1}{n}, \ldots, \frac{1}{n} \right)$
$$L_{\min} = L\left(\frac{1}{n}, \frac{1}{n}, \ldots, \frac{1}{n} \right) = \frac{1}{n}.$$
It follows that $ L \geq \frac{1}{n}$.

If $\gamma(M) = m$, $(m<n)$, then we are looking for the extremum of a function of $n-m$ variables. In this case, $L_{\min}= \frac{1}{n-m}$, so $L \geq \frac{1}{n-m}$.

From the problem's construction, it follows that the iteration of the transformation $\divideontimes$ does not change the minimum value of the function $L^{t}\big(p_{1}^{t}, \ldots, p_{n}^{t} \big)$, i.e., $ L^{t} \geq \frac{1}{n-m}.$

Therefore, for any $t$, the following holds: \begin{equation*} \frac{1}{n-m} \leq L^{t} \leq 1.\tag*{\qed}\end{equation*}\renewcommand{\qed}{}
\end{proof}

\begin{proposition} \label{pr0}
The sequence $L^{t} = \|{\bf p}^{t}\|^{2}$, $t = 0, 1, \ldots$ is convergent, i.e., there exists
$$ L^{\infty} = \lim_{t \to \infty} L^{t}.$$
\end{proposition}

\begin{proof}
It has been shown above that $\{L^{t}\}_{t=1}^{\infty}$ is a monotonic bounded sequence. Therefore, the sequence $\{L^{t}\}_{t=1}^{\infty}$ is convergent.
\end{proof}

\begin{figure}[h!]
\center{\includegraphics[scale=0.6]{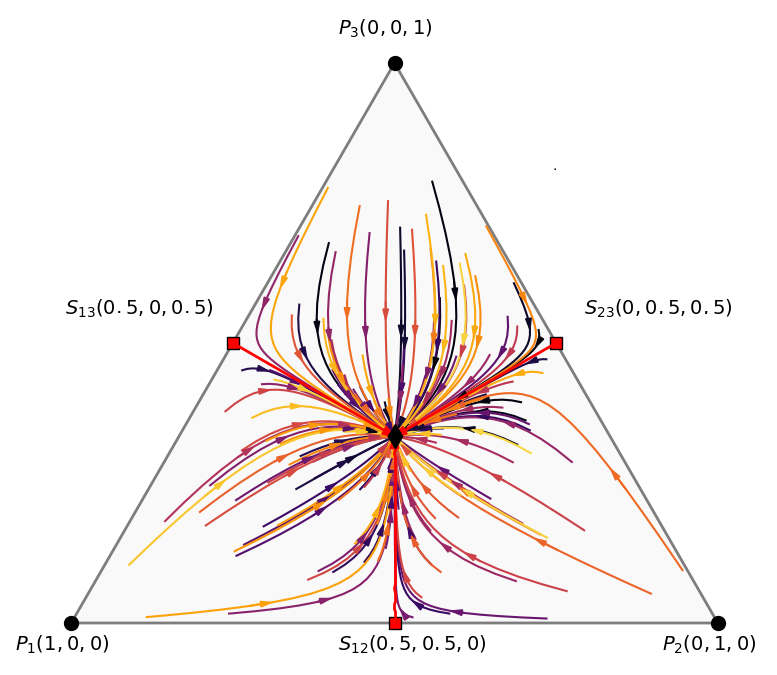}}
\caption{Phase portraits on the 2D invariant simplex within the 3D state space, illustrating the asymptotic trajectories of the stochastic vectors ${\bf p}^t$ under the action of conflict transformation~\eqref{eq:2}.
The trajectories asymptotically converge to the barycenter of the simplex $\left(\frac{1}{3}, \frac{1}{3}, \frac{1}{3} \right)$, which constitutes the interior attractor. The vertices of the simplex $P_1(1; 0; 0)$, $P_2(0; 1; 0)$, and $P_3(0; 0; 1)$ are unstable repelling fixed points.
Saddle points $S_{12}(0.5; 0.5; 0)$, $S_{23}(0; 0.5; 0.5)$, and $S_{13}(0.5; 0; 0.5)$ are located at the midpoints of the one-dimensional faces (edges) of the simplex. Their stable invariant manifolds coincide strictly with the boundary edges, capturing the phase flow from the repelling vertices. Concurrently, their unstable invariant manifolds extend into the interior of the phase space, defining the principal pathways of the flow toward the central attractor.
All other points within the simplex represent transient states.}
\label{fig-c1}
\end{figure}

\begin{proof}[Proof of Theorem \ref{th1pop}]
If some $i$-th coordinate~$p_{i}$ of the vector~${\bf p}$ is zero, then from Proposition \ref{pr2} it follows that $p_{i}^{t} = 0 = p_{i}^{\infty}.$

Consider a vector ${\bf p}$, all coordinates of which are strictly positive. Without loss of generality, we can assume that the coordinates of vector ${\bf p} $ are ordered in ascending order:
\begin{equation} \label{eq:4}
0 < p_{1} \leq p_{2} \leq \ldots \leq p_{n} < 1.
\end{equation}
From Proposition \ref{pr1} it follows that this ordering does not change, i.e., the order of coordinates does not depend on $t$.

Considering Proposition \ref{pr4}, all coordinates of the vector ${\bf p}$ at each discrete time moment $t$ are divided into three sets: those less than $ L^{t} $, those equal to $L^{t}$, and those greater than $L^{t}$. By Lemma \ref{l3}, the sequence of coordinates ${\lbrace p^{t}_{n} \rbrace}_{t = 0}^{\infty}$ is decreasing.
Let's show that the minimum coordinate~$ p^{t}_1$ always increases. Consider the difference
$$p^{t+1}_{1}-p^{t}_{1} = \frac{p_{1}^{t}\big(n-p_{1}^{t}\big)-p_{1}^{t}\big(n-L^{t}\big)}{n-L^{t}}=
\frac{p_{1}^{t}}{n-L^{t}}\big(L^{t}-p_{1}^{t}\big).$$
From Proposition \ref{pr4} it follows that $$p^{t+1}_{1}>p^{t}_{1}, \quad t = 0, 1, \ldots,$$
meaning the minimum coordinate increases.

Since the sequences ${\lbrace p_{1}^{t} \rbrace}_{t=0}^{\infty}$ and ${\lbrace p_{n}^{t} \rbrace}_{t=0}^{\infty}$ are monotonic and bounded, their limits exist:
$$ {p}_{\min}^{\infty} = {p}_{1}^{\infty} = \lim_{t\to \infty} {p}_{1}^{t}, \qquad {p}_{\max}^{\infty} ={p}_{n}^{\infty} = \lim_{t\to \infty}
{p}_{n}^{t}.$$

From equations \eqref{eq:3}, as $t \rightarrow \infty$, we get that $p_{1}^{\infty} = p_{n}^{\infty}= L^{\infty}$.

Considering condition \eqref{eq:4}, we have
\begin{equation}\label{eq:6}
p_{1}^{\infty} = p_{2}^{\infty} = \ldots = p_{j}^{\infty}= \ldots = p_{n-1}^{\infty}= p_{n}^{\infty}.
\end{equation}
Therefore, the limit ${\bf p}^{\infty} = \lim\limits_{t\to \infty} {\bf p}^t$ exists.
From the existence of the equilibrium vector ${\bf p}^{\infty} $ and equations \eqref{eq:2} it follows that $ {\bf r}^{\infty} = \lim\limits_{t \to \infty} {\bf r}^t$ exists.

Suppose $\gamma(M) = 0$. Considering \eqref{eq:6} and the stochasticity of the vectors, we have for
any $i = 1, \ldots, n$:
$$ p^{\infty}_{i} = r^{\infty}_{i} =\frac{1}{n} , $$
i.e., ${\bf p}^{\infty}={\bf r}^{\infty}.$

Let $0< \gamma(M) = m < n$, meaning $m$ coordinates of vector ${\bf p}$ are zero. Then the corresponding coordinates of vector~${\bf r}$ are $\frac{1}{n-1}.$ Consider the positive coordinates of vector~${\bf p}$. The number
of such coordinates is $n-m$.
From what has been proven above, it follows that
$$p^{\infty}_{i} = \frac{1}{n-m}.$$
Then the corresponding coordinates of vector ${\bf r}$ are
$$r^{\infty}_{i} = \frac{1-\frac{1}{n-m}}{n-1} = \frac{n-m-1}{(n-1)(n-m)}.$$

Let us recall that the dynamics are defined by the equation $p^{t+1}_i=p^{t}_i \cdot \frac {n-p^{t}_i}{n-L^{t}}$, where $L^t = \sum_{j=1}^n (p^t_j)^2$.

The calculation of the Jacobian matrix $J$ at the point ${\bf p}^* = \left(\frac{1}{n}, \ldots, \frac{1}{n}\right)$ yields a matrix with a very simple structure: all diagonal elements are identical ($J_{ii}=A$), and all off-diagonal elements are identical ($J_{ik}=B$). The Jacobian matrix $J$ has elements ${J_{ik} = \frac{\partial p^{t+1}_i}{\partial p^t_k}}$.

For $k=i$, we are looking for $\frac{\partial p^{t+1}_i}{\partial p^t_i}$. Let $p_i$ denote $p^t_i$ for convenience.
$$f_i(\mathbf{p}) = p_i \cdot \frac{n-p_i}{n-L}$$
where $L = \sum\limits_{j=1}^n p_j^2$.
Then
$$\frac{\partial f_i}{\partial p_i} = \frac{\partial p_i}{\partial p_i} \cdot \frac{n-p_i}{n-L} + p_i \cdot \frac{\partial}{\partial p_i} \left( \frac{n-p_i}{n-L} \right)$$

The first term is:
$$\frac{\partial p_i}{\partial p_i} \cdot \frac{n-p_i}{n-L} = 1 \cdot \frac{n-p_i}{n-L}$$

For the second term, we have
$$\frac{\partial}{\partial p_i} \left( \frac{n-p_i}{n-L} \right) = \frac{(-1)(n-L) - (n-p_i)(-2p_i)}{(n-L)^2} = \frac{-(n-L) + 2p_i(n-p_i)}{(n-L)^2}$$

Combining the terms, we get the expression for $A$:
$$A = \frac{n-p_i}{n-L} + p_i \cdot \frac{-(n-L) + 2p_i(n-p_i)}{(n-L)^2} = \frac{(n-p_i)(n-L) + p_i(-n+L+2np_i-2p_i^2)}{(n-L)^2}$$

Now we compute $A$ at the equilibrium point $\mathbf{p}^* = \left(\frac{1}{n}, \ldots, \frac{1}{n}\right)$.
At this point,
$p_i = \frac{1}{n}$.
$L = \sum\limits_{j=1}^n \left(\frac{1}{n}\right)^2 = n \cdot \left(\frac{1}{n^2}\right) = \frac{1}{n}$.
Then $A = \frac{n^3 - 2n + 2}{n(n^2-1)}.$

For $k \neq i$, we are looking for $\frac{\partial p^{t+1}_i}{\partial p^t_k}$
$$B = \frac{\partial f_i}{\partial p_k} = p_i(n-p_i) \frac{\partial}{\partial p_k} \left( \frac{1}{n-L} \right)$$
where
$$\frac{\partial}{\partial p_k} \left( \frac{1}{n-L} \right) = -\frac{-2p_k}{(n-L)^2} = \frac{2p_k}{(n-L)^2}$$

Now we combine for $B$:
$$B = p_i(n-p_i) \frac{2p_k}{(n-L)^2}$$

Let's compute $B$ at the equilibrium point ${\bf p}^* = \left(\frac{1}{n}, \ldots, \frac{1}{n}\right)$
\begin{align*} B &= \frac{1}{n} \left(n-\frac{1}{n}\right) \frac{2\left(\frac{1}{n}\right)}{\left(n-\frac{1}{n}\right)^2} = \frac{1}{n} \left(\frac{n^2-1}{n}\right) \frac{\frac{2}{n}}{\left(\frac{n^2-1}{n}\right)^2} \\ &= \frac{n^2-1}{n^2} \cdot \frac{2}{n} \cdot \frac{n^2}{(n^2-1)^2} = \frac{2(n^2-1)}{n(n^2-1)^2} = \frac{2}{n(n^2-1)}\end{align*}

Thus, the calculations show that:
$$A = \frac{n^3 - 2n + 2}{n(n^2-1)}, \quad B = \frac{2}{n(n^2-1)}$$

For matrices of such a structure, the eigenvalues are known: one value is $A+(n-1)B$, and the remaining $n-1$ values are $A-B$. We are interested in those $n-1$ eigenvalues that correspond to perturbations on the simplex. They are equal to
$$\lambda = A - B = \frac{n^3 - 2n + 2}{n(n^2-1)} - \frac{2}{n(n^2-1)} = \frac{n^3 - 2n}{n(n^2-1)} = \frac{n(n^2-2)}{n(n^2-1)} = \frac{n^2-2}{n^2-1}.$$
Since for $n>1$, $0 < n^2-2 < n^2-1$ holds, for this eigenvalue we have
$$0 < |\lambda| < 1$$
Since all $n-1$ relevant eigenvalues are less than unity in magnitude, the internal equilibrium point ${\bf p}^\infty = \left(\frac{1}{n}, \ldots, \frac{1}{n}\right)$ is locally asymptotically stable.

Consider the simplest boundary case where one coordinate is zero, for example, $p_n=0$, and the remaining $n-1$ coordinates are equal to $p_i^* = \frac{1}{n-1}$ for $i < n$.

For this point to be stable, it must be stable to perturbations that ``push'' it off the boundary. That is, if we add a small positive perturbation $\varepsilon$ to the zero coordinate $p_n$, it must decay rather than grow.

This stability is determined by the ``normal'' eigenvalue, which is equal to the diagonal element of the Jacobian $J_{nn}$, calculated at this boundary point.

Calculating the derivative $\frac{\partial F_n}{\partial p_n}$ at the point where $p_n=0$ and the other components are $\frac{1}{n-1}$, gives
$$\lambda_{\text{norm}} = J_{nn}({\bf p}^*) = \frac{n}{n - L^*} = \frac{n}{n - \sum_{j=1}^{n-1} \left(\frac{1}{n-1}\right)^2} = \frac{n}{n - (n-1)\frac{1}{(n-1)^2}}$$
$$  = \frac{n}{n - \frac{1}{n-1}} = \frac{n(n-1)}{n(n-1)-1} = \frac{n^2-n}{n^2-n-1}.$$

Since  $n^2-n > n^2-n-1$, it follows that $\lambda_{\text{norm}} > 1$.
An eigenvalue greater than unity means that any small perturbation of the zero component will grow exponentially, pushing the system away from this boundary. The boundary equilibrium point in this case is unstable.

Thus, stability analysis shows that only one equilibrium is stable: ${\bf p}^{\infty} ={\bf r}^{\infty} =\left(\frac{1}{n}, \frac{1}{n}, \ldots, \frac{1}{n}\right).$
\end{proof}
\begin{corollary}
For the dynamic conflict system \eqref{eq:1} with an initial stochastic vector ${{\bf p} \in{\mathbb R}_+^n}$, $(n>1)$, whose evolution is given by formulas \eqref{eq:3}, the set of all possible  vectors of equilibrium ${\bf p}^\infty \in \mathbb{R}_{+}^{n}$ has the following stability structure:
\begin{enumerate}
    \item The interior fixed point ${\bf p}^\infty = \left(\frac{1}{n}, \ldots, \frac{1}{n}\right)$ is a global attractor for all initial vectors ${{\bf p}^0= \{p_i^0\}_{i=1}^n}$, where $p_i^0 > 0$ for all $i$.
    \item The boundary of the simplex (the set of vectors where at least one coordinate is zero) is an invariant set. All fixed points belonging to this boundary are unstable in the full state space. Specifically, vertices act as unstable repellers, while fixed points located on the higher-dimensional faces act as topological saddles.
\end{enumerate}
\end{corollary}

The stability analysis of fixed points for nonlinear maps, as performed here via linearization and eigenvalue analysis of the Jacobian, is a cornerstone of dynamical systems theory \cite{Strogatz2018, El2005}. The result that boundary equilibria are unstable repellers while the interior equilibrium acts as an attractor is a common feature in systems modeling population dynamics and competition \cite{28}.

In addition to establishing stability, the analysis of the Jacobian's eigenvalues provides an important insight into the rate of convergence to the equilibrium. As shown, all $n-1$ eigenvalues corresponding to perturbations on the simplex are equal to $\lambda = \frac{n^2-2}{n^2-1}$. This value can be expressed as $\lambda = 1 - \frac{1}{n^2-1}$, from which it is evident that as the number of system components $n$ increases, $\lambda$ approaches 1 from below. Since the decay rate of perturbations near the fixed point is determined by the magnitude $|\lambda|^t$, the proximity of $\lambda$ to unity signifies that the convergence is very slow. Consequently, for systems with a large number of interacting components (a large $n$), the process of reaching a uniform distribution can be extremely prolonged. In practice, this means the system can spend long periods in transient, non-uniform states before reaching its asymptotic equilibrium.

\section{Dependence of the equilibrium on the set of constants.}\label{subsection342}
This heterogeneity in the parameters $c_i$ is conceptually similar to the introduction of influential agents or ``zealots'' in opinion dynamics, who can prevent a full consensus~\cite{Mobilia2007}. Whereas models of bounded confidence achieve fragmentation through interaction thresholds \cite{Hegselmann2002}, our model generates aggregated states analytically, based on the intrinsic ``attractiveness'' $c_i$ of each component.

Next, let's consider the case where all parameters $c_i$ are different. Mathematically, the equilibrium is entirely determined by this set of parameters~$c_i$. For instance, a violation of condition~\eqref{ccc} for some index~$i$ in Theorem~\ref{thpop2} leads to the corresponding coordinate vanishing, i.e., $p_i^\infty = 0$. This behavior has been confirmed by numerous computer simulations (see Fig.~\ref{ris:3pop_a},~\ref{ris:3pop_b}). Furthermore, the explicit form of the vectors of equilibrium demonstrates that a greater value of the constant $c_i$ strictly correlates with a greater value of the corresponding $i$-th coordinate.

In terms of the applied model, this mathematical dependence illustrates how external parameters (interpreted as assistance or favorability) govern the resource distribution \cite{Karataieva2025}. Specifically, more favorable conditions associated with a certain state~$\omega_i$ lead to an increase in the corresponding component's share in that state. Conversely, the condition $p_i^\infty = 0$ is interpreted as the complete depletion of the resource from the $i$-th state, resulting in the component's displacement from the system.

Considering that ${r^t_i = \frac{1-p_i^t}{n-1}}$, equation \eqref{eq:3.1} can be written as
\begin{equation}\label{eq:3.2}
p^{t+1}_i= p_{i}^{t} \cdot \frac{n-1 + c_{i}(1-p_{i}^{t})}{n-1+L_{c}^{t}} = p_{i}^{t} \cdot k_{i,c}^{t}, \quad t \geq 1,
\end{equation}
where
$$k_{i,c}^{t} = \frac{n-1 + c_{i}(1-p_{i}^{t})}{n-1+L_{c}^{t}}, \quad L_c^t = \sum_{i=1}^{n}c_{i}p^t_{i}(1-p^t_{i}).$$
Hereafter, we will use the notation $\kappa_{i}^{t} := c_i(1-p_{i}^{t})$.
\begin{proposition}
Let for some $i$, $j$ $(i \neq j)$ the parameter $c_i$ be equal to parameter $c_j$, i.e.,
$c_i=c_j=c$, and $p^t_i<p^t_j$.
Then
$$\kappa^t_i > \kappa^t_j \Leftrightarrow \kappa^{t+1}_i > \kappa^{t+1}_j.$$
\end{proposition}

\begin{proof}
The inequalities $\kappa^t_i > \kappa^t_j$ and $p^t_i<p^t_j$ are equivalent, since
$$\kappa^{t}_{i} - \kappa^t_j = c(1-p_i^t) - c(1-p_j^t) = c(p^t_j-p^t_i) >0.$$
Therefore, it is sufficient to show that $p^t_i<p^t_j$ implies $p^{t+1}_i<p^{t+1}_j$. Consider
\begin{gather*} \begin{split}
p^{t+1}_j - p^{t+1}_i &= \frac{1}{n-1 + L_c}\big(p^t_j(n-1+c(1-p^t_j))- p^t_i(n-1+c(1-p^t_i)) \big) \\
&= \frac{1}{n-1 + L_c}\big((p^t_j-p^t_i)(n-1+c)- c(p^t_j)^2 + c(p^t_i)^2) \big) \\
&= \frac{1}{n-1 + L_c}\big((p^t_j-p^t_i)(n-1+c)+ c(p^t_i - p^t_j)(p^t_j + p^t_i)\big) \\
&= \frac{p^t_j-p^t_i}{n-1 + L_c}\big(n-1+c-c(p^t_j + p^t_i)\big) \\
&=\frac{p^t_j-p^t_i}{n-1 + L_c}\big(n-1+c(1-p^t_j - p^t_i)\big)>0,
\end{split}\end{gather*}
thus $p^{t+1}_i < p^{t+1}_j$.

Hence, if $\kappa^t_i > \kappa^t_j$, then $\kappa^{t+1}_i > \kappa^{t+1}_j$.
\end{proof}

Let $V_n$ be the set of vertices of the standard $(n-1)$-dimensional simplex in $\mathbb R^n$. This set consists of the $n$ standard basis vectors:
$$V_n = \{{\bf e}_1, {\bf e}_2, \ldots, {\bf e}_n\},$$
where ${\bf e}_i = (0, \ldots, 1, \ldots, 0)$ is the vector with a 1 in the $i$-th position and zeros elsewhere.

\begin{proposition}
  Every vertex ${\bf p} \in V_n$ is a fixed point of the dynamical system
\end{proposition}

\begin{proof}
  Let the initial state of the system be ${\bf p}^0 = {\bf e}_i$ for some $i \in \{1, \ldots, n\}$. This means the initial components of the system are $p_i^0 = 1$ and $p_j^0 = 0$ for all $j \neq i$.

Consider the evolution of an arbitrary component $p_j^t$ according to the dynamical equation ${\bf p}_j^{t+1} = {\bf p}_j^{t} \cdot k_{j,c}^{t}$.
 For any index $j \neq i$, the initial component is $p_j^0 = 0$. Therefore, in the first step, $p_j^1 = 0 \cdot k_{j,c}^{0} = 0$. By mathematical induction, it is evident that $p_j^t = 0$ for all $t \ge 0$.

 Since the sum of the components of the state vector must equal one for all time $t$ ($\sum_{k=1}^{n} p_k^t = 1$), and we have proven that all components $p_j^t$ for $j \neq i$ are zero, the component $p_i^t$ must be equal to 1 for all $t \ge 0$.

Thus, if ${\bf p}^0 = {\bf e}_i$, then ${\bf p}^t = {\bf e}_i$ for all $t \ge 0$, which proves that every vertex is a fixed point.
\end{proof}

\subsection{Existence of a fixed point}
\begin{theorem}\label{thpop1}
Let ${\bf p}^0 \in \mathbb R_+^{2}$, ${\bf p}^0 \neq {\bf e}_i$. Then each trajectory of the dynamic conflict system \eqref{eq:1} converges to a fixed point
$$ {\bf p}^{\infty} = \lim_{t\to \infty} {\bf p}^t,$$
and
\begin{equation}\label{eq1r}
{\bf p}^{\infty} = \left(\frac{c_1}{c_1+c_2}, \frac{c_2}{c_1+c_2}\right).
\end{equation}
\end{theorem}

\begin{proof}
Let ${\bf p} = (p_1, p_2)$ and ${\bf c}=(c_1, c_2)$. Obviously, ${{\bf p} = (1;0) = {\bf p}^{\infty}}$, and similarly ${\bf p} = (0;1) = {\bf p}^{\infty}$.

Consider $p_i>0$, $i =1, 2$.
Assume $\frac{c_1}{c_2} < \frac{p_1}{p_2}$. Then
\begin{gather*} \begin{split}
p^1_1 -p_1 &= \frac{p_1}{1+L_c}\big( c_1(1-p_1)-c_1p_1(1-p_1)-c_2p_2(1-p_2) \big) \\
&= \frac{p_1}{1+L_c}\big( c_1(1-p_1)^2-c_2p_2(1-p_2) \big) \\
&= \frac{p_1}{1+L_c}\big( c_1p_2^2-c_2p_2p_1 \big) = \frac{p_1p_2}{1+L_c}\big( c_1p_2-c_2p_1 \big) <0,
\end{split}\end{gather*}
thus $p^1_1 < p_1$. Then $p_2^1 = 1-p^1_1 > 1-p_1 = p_2.$ From this it follows that $\frac{c_1}{c_2} < \frac{p_1}{p_2}< \frac{p^1_1}{p^1_2}$.
Thus, we get
$\frac{c_1}{c_2} < \frac{p_1^t}{p_2^t} \quad \forall \ t.$

Analogously for the case $\frac{c_1}{c_2} > \frac{p_1}{p_2}$.

Hence, $\{p^t_i\}_{t=0}^{\infty}$ is a monotonic bounded sequence, meaning the limit exists
$$ {\bf p}^{\infty} = \lim_{t\to \infty} {\bf p}^t.$$
Then from the invariance of the equilibrium, we have the system
\begin{equation*}
\begin{cases}
k_{1,c}^{\infty}=1,\\
k_{2,c}^{\infty}=1,
\end{cases}
\Longleftrightarrow
\begin{cases}
c_1(1-p^{\infty}_1)=L_c^{\infty},\\
c_2(1-p^{\infty}_2)=L_c^{\infty},
\end{cases}
\Longleftrightarrow
\begin{cases}
p^{\infty}_1=\dfrac{c_1}{c_1+c_2},\\
p^{\infty}_2=\dfrac{c_2}{c_1+c_2},
\end{cases}
\end{equation*}
which proves equality \eqref{eq1r}.
\end{proof}

Next, the case of a three-dimensional space is considered, i.e., ${\bf p} \in \mathbb R_+^{3}$.
Note that Theorem \ref{thpop1}, as well as the following Theorem \ref{thpop2}, can be generalized to the case ${\bf p} \in \mathbb R_+^{n}$, $n>3$. In this case, condition \eqref{ccc} becomes significantly more complex, and the description of the vector ${\bf p}^\infty$ breaks down into several non-trivial cases that require additional analysis.

\begin{figure}[h!]
\renewcommand{\thefigure}{2a}
\centering
\includegraphics[width=0.7\linewidth]{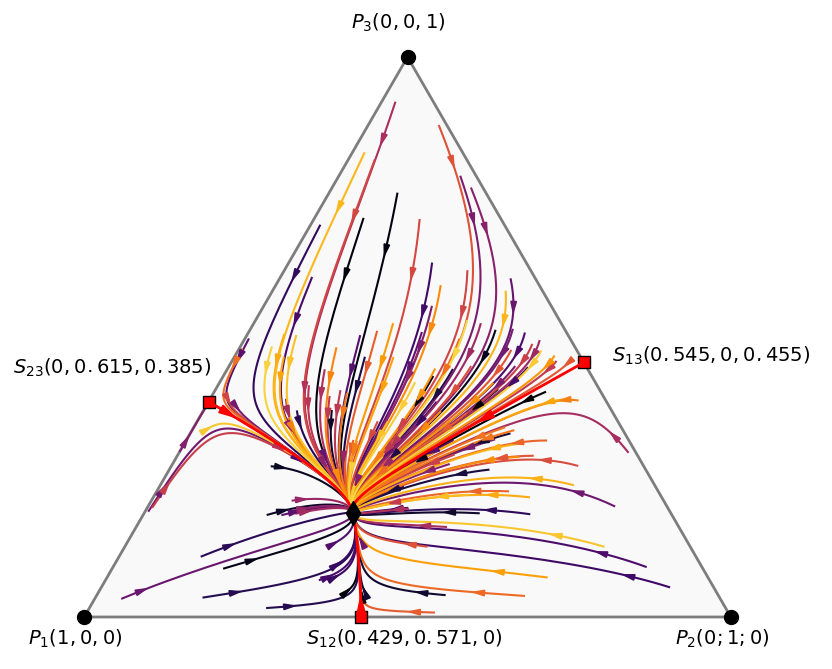}
\caption{Phase portrait on the 2D invariant simplex within the 3D state space, illustrating the asymptotic trajectories of the stochastic vectors ${\bf p}^t$ under the action of conflict transformation~\eqref{eq:3.2}. Case (a) with asymmetric interaction constants ${\bf c} = (0.3; 0.4; 0.25)$. The trajectories asymptotically converge to the interior attractor ${\bf p}^\infty = (0.322; 0.4915; 0.1865)$. The boundary fixed points act as saddles: $S_{12}(0.429; 0.571; 0)$, $S_{23}(0; 0.615; 0.385)$, and $S_{13}(0.545; 0; 0.455)$, with their unstable invariant manifolds defining the principal pathways of the flow toward the interior attractor. The vertices $P_1(1; 0; 0)$, $P_2(0; 1; 0)$, and $P_3(0; 0; 1)$ are unstable repelling fixed points, and all other points within the simplex represent transient states.}
\label{ris:3pop_a}
\end{figure}

\begin{theorem}\label{thpop2}
Let ${\bf p}^0 \in \mathbb R_+^{3}$, ${\bf p}^0 \neq {\bf e}_i$, and for each $i=1, \ldots, 3$ one of the conditions $\kappa^t_i > L_c^t$ or $\kappa^t_i < L_c^t$ holds for any $t$. Then every trajectory of the dynamic conflict system \eqref{eq:1} converges to a fixed point
$$ {\bf p}^{\infty} = \lim_{t\to \infty} {\bf p}^t.$$
Furthermore,
$${{\bf p}^{\infty}} = \bigg(\frac{c_1c_2+c_1c_3-c_2c_3}{c_1c_2+c_1c_3+c_2c_3}, \frac{c_1c_2-c_1c_3+c_2c_3}{c_1c_2+c_1c_3+c_2c_3}, \frac{c_1c_3-c_1c_2+c_2c_3}{c_1c_2+c_1c_3+c_2c_3}\bigg),$$
if for all $i, j, k = 1, 2, 3$, $i \neq j \neq k$
\begin{equation} \label{ccc}
c_i > \frac{c_j c_k}{c_j +c_k}.
\end{equation}
\end{theorem}

\begin{figure}[h!]
\renewcommand{\thefigure}{2b}
\addtocounter{figure}{-1}
\centering
\includegraphics[width=0.8\linewidth]{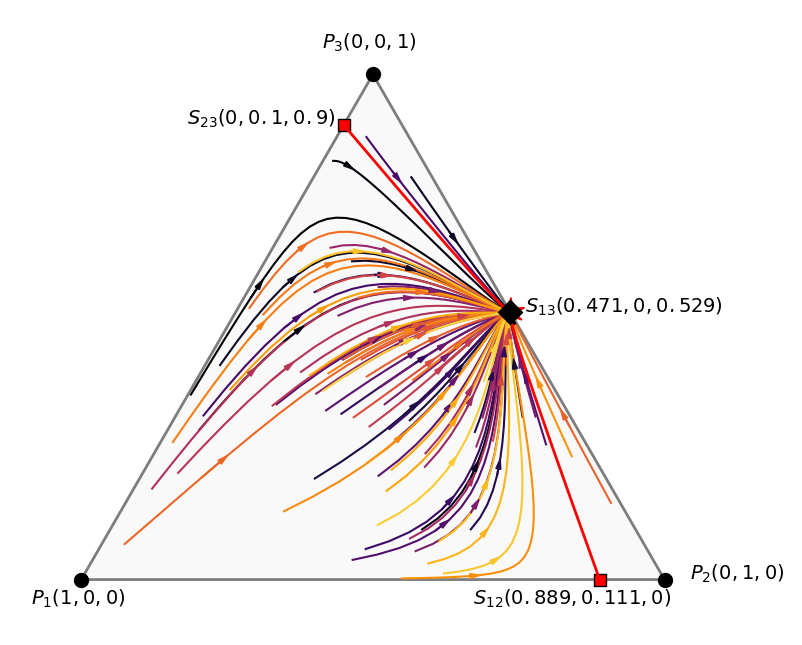}
\caption{Phase portrait on the 2D invariant simplex within the 3D state space. Case (b) with ${\bf c} = (0.8; 0.1; 0.9)$. The condition for the existence of an interior fixed point is violated (${c_2 < \frac{c_1 c_3}{c_1 + c_3}}$), inducing a transcritical bifurcation.  Consequently, the boundary fixed point exchanges stability with the former interior equilibrium, becoming the interior attractor $S_{13}$ corresponding to ${\bf p}^\infty = (0.471; 0; 0.529)$, which implies the asymptotic vanishing of the second coordinate ($p_2^\infty = 0$). The remaining saddles $S_{12}(0.889; 0.111; 0)$ and $S_{23}(0; 0.1; 0.9)$ generate unstable invariant manifolds that guide the entire phase flow directly into $S_{13}$. As in case (a), the vertices are unstable repelling fixed points.}
\label{ris:3pop_b}
\end{figure}

\begin{proof}
The condition $\kappa^t_i > L_c^t \quad (\kappa^t_i < L_c^t)$ is necessary and sufficient for the increase (decrease) of $p^t_i$, i.e.,
$$\kappa^t_i > L_c^t \Leftrightarrow p_i^{t+1} > p^t_i \quad (\kappa^t_i < L_c^t \Leftrightarrow p_i^{t+1} < p^t_i).$$
This follows directly from equations \eqref{eq:3.2}. If this condition holds for all $t$, then the sequence $\{ p_i^t\}_{t=0}^\infty$ is monotonic. The boundedness of this sequence follows from the problem statement. Therefore, an asymptotic state ${\bf p}^{\infty}$ exists. Assume all coordinates $p_i^\infty$ are non-zero. Then the invariance of the equilibrium means that $k_{i,c}^{t} =1$ for all $i= 1, 2, 3$. That is, we have the system
 \begin{equation*}
 \begin{cases}
 c_1(1-p_1^{\infty})^2=c_2 p_2^{\infty}(1-p_2^{\infty})+c_3 p_3^{\infty}(1-p_3^{\infty}),\\
 c_2(1-p_2^{\infty})^2=c_1 p_1^{\infty}(1-p_1^{\infty})+c_3 p_3^{\infty}(1-p_3^{\infty}),\\
 c_3(1-p_3^{\infty})^2=c_1 p_1^{\infty}(1-p_1^{\infty})+c_2 p_2^{\infty}(1-p_2^{\infty}),
 \end{cases}
\end{equation*}
from which we obtain
 \begin{equation}\label{eqinf}
 \begin{cases}
 p_1^\infty = \dfrac{c_1(c_2+c_3)-c_2c_3}{c_1c_2+c_1c_3+c_2c_3},\vspace{1mm}\\
p_2^\infty = \dfrac{c_2(c_1+c_3)-c_1c_3}{c_1c_2+c_1c_3+c_2c_3},\vspace{1mm}\\
 p_3^\infty = \dfrac{c_3(c_1+c_2)-c_1c_2}{c_1c_2+c_1c_3+c_2c_3}.
 \end{cases}
\end{equation}
 By problem statement, all coordinates of the vector ${\bf p}^t$ are non-negative. Therefore, equalities \eqref{eqinf} are valid only when $c_i > \frac{c_j c_k}{c_j +c_k}$ for all $i \neq j \neq k $, $i, j, k = 1, 2, 3$.

It should be noted that the equilibrium is stable, as it is determined only by a fixed set of constant parameters $c_i$, which in turn are given at $t=0$ and do not depend on time~$t$.

If for some $i$ the inequality $c_i \leq \frac{c_j c_k}{c_j +c_k}$ holds, then $$p_i^\infty = \frac{c_i(c_j+c_k)-c_jc_k}{c_ic_j+c_ic_k+c_jc_k} \leq \frac{\frac{c_j c_k}{c_j +c_k}(c_j+c_k)-c_jc_k}{c_ic_j+c_ic_k+c_jc_k} = \frac{{c_j c_k}-c_jc_k}{c_ic_j+c_ic_k+c_jc_k}=0,$$
but the case $p_i^\infty < 0 $ is impossible, so $p_i^\infty = 0 $. Hence, if for some $i$-th index condition \eqref{ccc} is violated, then the corresponding coordinate~$p_i^t$ tends to $0$ as $t \rightarrow \infty$.
 \end{proof}

Let $M^* = \{k: p_k^\infty > 0\}$ be the set of indices for which the coordinates of the equilibrium vector are strictly positive. Then $\gamma(M^*)$ is the number of elements in the set $M^*$. The existence and uniqueness of $M^*$ are determined by the iterative procedure~\eqref{eq:1}. Let's denote
$$\Lambda(M^*) = \frac{\gamma(M^*)-1}{\sum_{k \in M^*} \frac{1}{c_k}}$$
According to~\eqref{eq:3.2}, if $p_j^\infty=1$ for a unique $j \in M^*$, then $\Lambda(M^*)=0$; in all other cases, $\Lambda(M^*)>0$.

\begin{theorem}\label{thpop3}
 Let ${\bf p} \in \mathbb{R}^n$, $n > 3$. Suppose that for each $i = 1, \ldots, n$, starting from some time $t$, one of the conditions $k_i^{t} > L_c^{t}$ or $k_i^{t} < L_c^{t}$ holds. Then every trajectory of the dynamic conflict system \eqref{eq:1} converges to a fixed point
$$
{\bf p}^\infty = \lim_{t \to \infty} {\bf p}^t.
$$
For any $i=1, \ldots, n$, one of the following equalities holds:
\begin{enumerate}
    \item $p_i^\infty = 0$, if $c_i \le \Lambda(M^*)$.
    \item $p_i^\infty = 1 - \frac{\Lambda(M^*)}{c_i}$, if $c_i > \Lambda(M^*)$.
\end{enumerate}
\end{theorem}

 \begin{proof}
Similarly to the proof of Theorem~\ref{thpop2} for $n=3$, based on the monotonicity of the sequences $\{p_i^t\}$, the trajectories ${\bf p}^t$ converge to an equilibrium vector~${\bf p}^\infty$, which is a fixed point of the system.

For a fixed point ${\bf p}^\infty$, if $p_i^\infty > 0$ (i.e., $i \in M^*$) and $c_i > 0$, it follows from the dynamical equation~\eqref{eq:3.2} $p_{i}^{t+1}=p_{i}^{t}\cdot k_{i,c}^{t}$, where $k_{i,c}^{t}=\frac{n-1+c_{i}(1-p_{i}^{t})}{n-1+L_{c}^{t}}$, that in the stationary state $k_{i,c}^{\infty}=1$.
This means that the numerator and denominator of $k_{i,c}^{\infty}$ are equal:
$$n-1+c_{i}(1-p_{i}^{\infty})=n-1+L_{c}^{\infty}.$$
Hence $c_i(1-p_i^\infty) = L_c^\infty$. Let's denote $L_c^\infty =: \Lambda(M^*)$, as this value depends on the set of indices $M^*$ for which $p_k^\infty > 0$ and which contribute to $L_c^\infty$.
Thus, for all $i \in M^*$, where $c_i>0$, we have
$$p_i^\infty = 1 - \frac{\Lambda(M^*)}{c_i}.$$

The sum of all coordinates of the vector ${\bf p}^\infty$ is equal to 1. Since $p_k^\infty = 0$ for $k \notin M^*$, then
$$\sum_{k \in M^*} p_k^\infty = 1.$$
Substituting the expression for $p_k^\infty$:
$$\sum_{k \in M^*} \left(1 - \frac{\Lambda(M^*)}{c_k}\right) = 1$$
$$\gamma(M^*) - \Lambda(M^*) \sum_{k \in M^*} \frac{1}{c_k} = 1$$
From this, if $\gamma(M^*) \ge 1$ and the denominator is not zero (which holds if at least one $c_k \neq 0$ for $k \in M^*$), we get:
$$\Lambda(M^*) = \frac{\gamma(M^*)-1}{\sum_{k \in M^*} \frac{1}{c_k}}$$
If $\gamma(M^*)=1$, say $M^*=\{j\}$, then $p_j^\infty=1$. In this case, $1 - \Lambda(\{j\})/c_j = 1$, which means $\Lambda(\{j\})=0$. The formula yields $(1-1)/ (1/c_j) = 0$.

For $j \notin M^*$, we have $p_j^\infty = 0$.
If $c_j > 0$, then the condition $p_j^\infty=0$ means that if we formally calculated $1 - \frac{\Lambda(M^*)}{c_j}$, the result should be non-positive, i.e., $1 - \frac{\Lambda(M^*)}{c_j} \le 0$, which is equivalent to $c_j \le \Lambda(M^*)$.
If $c_j=0$, then from equation~\eqref{eq:3.1} $p_{j}^{t+1}=\frac{p_{j}^{t}(1+c_{j}r_{j}^{t})}{z^{t}} = \frac{p_j^t}{z^t}$. If $z^\infty > 1$ (which occurs if $\Lambda(M^*) > 0$ or $L_c^\infty > 0$), then $p_j^t \rightarrow 0$. If $z^\infty = 1$ (for example, if all $c_k=0$), then $p_j$ does not change. However, in the context of the theorem, where $c_i$ affect the distribution, it is assumed that not all $c_i$ are trivial. The condition $c_j=0 \le \Lambda(M^*)$ (if $\Lambda(M^*) \ge 0$) also aligns with $p_j^\infty=0$.

The existence and uniqueness of the set $M^*$, which satisfies these self-consistent conditions ($c_i > \Lambda(M^*)$ for $i \in M^*$ and $c_j \le \Lambda(M^*)$ for $j \notin M^*$, $c_j>0$), are proven in similar models using the iterative procedure described in the theorem's formulation. The procedure converges because at each step the set $M^*$ either remains unchanged or decreases, and the number of indices is finite.

Thus, for $i \in M^*$
$$p_i^\infty = 1 - \frac{\Lambda(M^*)}{c_i} > 0.$$
For $j \notin M^*$, $p_j^\infty = 0$, which holds under the condition $c_j \le \Lambda(M^*)$.
Therefore, the explicit form of the equilibrium depends on the set $M^*$, which, in turn, is determined by the entire set of constants $c_k$.
\end{proof}

 The condition that for each $i$, the sign of $\kappa^t_i - L_c^t$ remains  constant for all $t$, is a sufficient condition to ensure the monotonicity of  each coordinate sequence $\{p_i^t\}_{t=0}^\infty$, which greatly simplifies the  proof of convergence. While this condition holds for a wide range of initial states, investigating the system's dynamics in cases where this sign may oscillate presents a complex case for the proof of convergence. This work focuses on the class of trajectories exhibiting such monotonic behavior.

\begin{remark}
If the condition
\begin{equation}\label{eneq_all}
    c_k > \frac{n-1}{\sum\limits_{j=1, j \neq k}^n \frac{1}{c_j}}
\end{equation}
holds for all $k=1, \ldots, n$, then $\gamma(M^*)=n$,
$$\Lambda(M^*)= \frac{n-1}{\sum\limits_{j=1}^n \frac{1}{c_j}},$$
and for all $i=1, \ldots, n$ we have:
$$p_i^\infty = 1 - \frac{\Lambda}{c_i} = 1 - \frac{n-1}{c_i \sum\limits_{j=1}^n \frac{1}{c_j}} > 0.$$
\end{remark}

\begin{remark}
As the mean--field vector ${\bf r}^t$ is not an independent variable, being defined at each step
as a function of ${\bf p}^t$ by formula~\eqref{eq:3.1r}, its equilibrium~${\bf r}^\infty$
is also entirely dependent on ${\bf p}^\infty$. Therefore, after the coordinates~$p_i^\infty$
have been determined from Theorems~\ref{thpop1}--\ref{thpop3}, the corresponding coordinates~$r_i^\infty$
can be readily computed as:
$$r^{\infty}_i = \frac{1 - p^{\infty}_i}{n-1}.$$
\end{remark}

Below, we present the formal mathematical procedure to analytically compute the fixed point ${\bf p}^\infty$ for a given initial vector ${\bf p}^{t=0}$ and a set of control parameters $c_i$, $i =1, \ldots, n$. The complete programmatic implementation of this algorithm, available as \emph{find\_fixed\_point.py}, can be found in the public repository (see~\ref{sec:code}).

\noindent \textbf{Input:} Initial state vector ${\bf p}^0 \in \mathbb{R}^n_+$, control parameters $c_1, \ldots, c_n > 0$. \\
\textbf{Output:} The equilibrium vector ${\bf p}^\infty \in \mathbb{R}^n_+$.
\begin{description}
    \item[Step 1.] \emph{Initialization (Filtering trivial components):} Determine the initial set of active components that participate in the competition:
    $$M_0 = \{k \in \{1, \ldots, n\} \mid p_k^0 > 0\}.$$
    Set the initial candidate set $M_{cand} = M_0$.

    \item[Step 2.] \emph{Iterative thresholding (Finding the equilibrium set $M^*$):} Update the candidate set $M_{cand}$ iteratively:
    \begin{enumerate}
        \item[(2.a)] Calculate the critical threshold $\Lambda(M_{cand})$ for the current set:
        $$ \Lambda(M_{cand}) = \begin{cases} 0, & \text{if } |M_{cand}| \le 1, \\ \frac{|M_{cand}| - 1}{\sum\limits_{k \in M_{cand}} \frac{1}{c_k}}, & \text{otherwise.} \end{cases} $$
        \item[(2.b)] Identify the subset of indices that satisfy the strict survival condition:
        $$ M_{new} = \{k \in M_{cand} \mid c_k > \Lambda(M_{cand})\}. $$
        \item[(2.c)] Check for stabilization: if $M_{new} = M_{cand}$, the iteration terminates. Set the final equilibrium set $M^* = M_{cand}$ and proceed to Step 3. Otherwise, update $M_{cand} = M_{new}$ and repeat Step 2.
    \end{enumerate}

    \item[Step 3.] \emph{Computation of the final state:} Calculate the exact coordinates of the fixed point ${\bf p}^\infty$ for all $i = 1, \ldots, n$:
    $$ p_i^\infty = \begin{cases} 1 - \frac{\Lambda(M^*)}{c_i}, & \text{if } i \in M^*, \\ 0, & \text{if } i \notin M^*. \end{cases} $$
\end{description}

\subsection{Analysis of the stability of equilibria}
\textbf{Case $n=2$}. Theorem~\ref{thpop1} states that for a system in the space $\mathbb{R}_{+}^{2}$, the trajectories converge to the fixed point
$$
{\bf p}^{\infty} = (p_1^{\infty}, p_2^{\infty}) = \left(\frac{c_1}{c_1+c_2}, \frac{c_2}{c_1+c_2}\right)
$$
In the previous subsection, we proved convergence to this point. We will now prove its stability.

Since $p_1 + p_2 = 1$, the dynamical system is one-dimensional. The dynamics can be expressed in terms of a single variable, for instance $p_1$. Let $p_1^t = x$, which implies $p_2^t = 1-x$.

The dynamical equation for $p_1^{t+1}$ is given by:
$$
p_1^{t+1} = f(p_1^t) = \frac{p_1^t(1+c_1 r_1^t)}{z^t}
$$
$$r_1^t = \frac{1-p_1^t}{2-1} = 1-p_1^t, \quad r_2^t = \frac{1-p_2^t}{2-1} = 1-(1-p_1^t) = p_1^t,$$
$$z^t = p_1^t(1+c_1 r_1^t) + p_2^t(1+c_2 r_2^t) = p_1^t(1+c_1(1-p_1^t)) + (1-p_1^t)(1+c_2 p_1^t).$$

Thus, our map function is:
\begin{equation}\label{x=2_func}
f(x) = \frac{x(1+c_1(1-x))}{x(1+c_1(1-x)) + (1-x)(1+c_2 x)}
= \frac{x + c_1 x - c_1 x^2}{1+(c_1+c_2)x-(c_1+c_2)x^2}
\end{equation}
The fixed point $x^{\infty} = p_1^{\infty} = \frac{c_1}{c_1+c_2}$ is the solution to the equation $f(x) = x$. For stability analysis, we need to find the derivative $f'(x)$ and evaluate it at the point $x^{\infty}$. For stability, we require $|f'(x^{\infty})| < 1$.

At the fixed point ${\bf p}^{\infty}$, the condition $c_i(1-p_i^\infty) = L_c^\infty$ holds for all $i$ where $p_i^\infty > 0$, and $L_c^\infty$ is
$$L_c^\infty = c_1 p_1^\infty (1-p_1^\infty) + c_2 p_2^\infty (1-p_2^\infty) = c_1 p_1^\infty p_2^\infty + c_2 p_2^\infty p_1^\infty = p_1^\infty p_2^\infty (c_1+c_2)$$

Substituting the values of $p_1^\infty$ and $p_2^\infty$:
\begin{equation*}
L_c^\infty = \frac{c_1}{c_1+c_2} \cdot \frac{c_2}{c_1+c_2} \cdot (c_1+c_2) = \frac{c_1 c_2}{c_1+c_2}
\end{equation*}

To find the derivative $f'(x)$, we apply the quotient rule, which leads to the following complex fraction:
\begin{equation*}
\begin{split}
    f'(x) = \frac{1}{[1 + (c_1+c_2)x - (c_1+c_2)x^2]^2} \Biggl( & (1 + c_1 - 2c_1x) \cdot [1 + (c_1+c_2)x - (c_1+c_2)x^2] \\
    & - (x + c_1x - c_1x^2) \cdot [c_1 + c_2 - 2(c_1+c_2)x] \Biggr)
\end{split}
\end{equation*}
We will evaluate $f'(x)$ at the fixed point $x^{\infty} = \frac{c_1}{c_1+c_2}$.

At the fixed point, the numerator of the function~\eqref{x=2_func} is equal to the product of the denominator and $x^\infty$:
$$ x^{\infty} + c_1x^{\infty} - c_1(x^{\infty})^2 = x^{\infty} \cdot [1 + (c_1+c_2)x^{\infty} - (c_1+c_2)(x^{\infty})^2] $$
This allows us to simplify the expression for the derivative at the point $x^{\infty}$:
\begin{equation*}
f'(x^{\infty})= \frac{ (1 + c_1 - 2c_1x^{\infty}) - x^{\infty} \cdot [c_1 + c_2 - 2(c_1+c_2)x^{\infty}] }{ 1 + (c_1+c_2)x^{\infty} - (c_1+c_2)(x^{\infty})^2 }
\end{equation*}
Now, let's consider only the numerator of the resulting fraction
\begin{align*}
c_1 + c_2 - 2(c_1+c_2)x^{\infty} &= c_1 + c_2 - 2(c_1+c_2)\left(\frac{c_1}{c_1+c_2}\right) \\
&= c_1 + c_2 - 2c_1 = c_2 - c_1
\end{align*}
Then
\begin{align*}
(1 + c_1 - 2c_1x^{\infty}) - x^{\infty}(c_2 - c_1)
&= 1 + c_1 - 2c_1x^{\infty} - c_2x^{\infty} + c_1x^{\infty}= 1 + c_1 - x^{\infty}(c_1 + c_2) \\
&= 1 + c_1 - \left(\frac{c_1}{c_1+c_2}\right)(c_1 + c_2) = 1 + c_1 - c_1 = 1
\end{align*}
Thus, the entire numerator is equal to 1. Returning to the expression for $f'(x^\infty)$, we get:
\begin{equation}
f'(x^{\infty}) = \frac{1}{1 + (c_1+c_2)x^{\infty} - (c_1+c_2)(x^{\infty})^2}
\end{equation}
The denominator is the value $z^\infty = 1+L_c^\infty$. Thus, the result is proven:
\begin{equation*}
f'(x^{\infty}) = \frac{1}{1+L_c^\infty}
\end{equation*}
Therefore,
\begin{equation*}
f'(p_1^\infty) = \frac{1}{1 + \frac{c_1 c_2}{c_1+c_2}} = \frac{c_1+c_2}{c_1+c_2+c_1c_2}
\end{equation*}
Since $c_1$, $c_2$ are positive constants for an internal fixed point, it follows that $c_1c_2 > 0$. Consequently, the denominator $(c_1+c_2+c_1c_2)$ is strictly greater than the numerator $(c_1+c_2)$.

Thus, the derivative at the fixed point satisfies the condition:
\begin{equation}
0 < f'(p_1^\infty) < 1
\end{equation}
Since $|f'(p_1^\infty)| < 1$, the fixed point ${\bf p}^{\infty}$ is locally asymptotically stable.

It is also necessary to analyze the stability of the boundary fixed points $(1;0)$ and $(0;1)$. The stability of these points is determined by the derivative of the map $f(x)$ at $x=1$ and $x=0$ respectively. A direct calculation yields:
\begin{equation*}
    f'(1) = 1+c_2 \quad f'(0) = 1+c_1
\end{equation*}
Since $c_1, c_2 > 0$, the modulus of the derivative at both points is strictly greater than 1. This proves that the boundary fixed points $(1;0)$ and $(0,1)$ are unstable repellers. Any small perturbation away from these ``pure'' states will cause the system to evolve towards the internal stable fixed point ${\bf p}^{\infty}$.

\begin{theorem}
For the dynamic conflict system \eqref{eq:1} at case $n=2$ the following stability properties hold:

    The point
    $${\bf p}^{\infty} = \left(\frac{c_1}{c_1+c_2}, \frac{c_2}{c_1+c_2}\right)$$
    is \emph{locally asymptotically stable}. It acts as the unique attractor of the system.

    The points ${\bf e}_1 = (1,0)$ and ${\bf e}_2 = (0,1)$ are \emph{unstable repellers}.
\end{theorem}

\textbf{Case $n \geq 3$}. In the stability analysis for $n \ge 3$, we focus on initial state vectors ${\bf p}^{t=0}$ with strictly positive coordinates, i.e. points in the interior of the $(n-1)$-dimensional simplex.
If some coordinate of the initial vector ${\bf p}^{t=0}$ is exactly zero, the trajectory remains forever on the corresponding boundary face in the absence of perturbations, since this face is an invariant set of the dynamics.
However, any arbitrarily small transversal (out-of-face) perturbation makes this coordinate positive and shifts the system's state into a higher-dimensional face of the simplex.
The equilibrium ${\bf p}^\infty$ in this perturbed case will, in general, differ from the boundary equilibrium, which implies that such a boundary fixed point is unstable with respect to transversal perturbations, even if it is stable with respect to perturbations tangential to the face.

Let us consider the map $\mathbf{f}: \mathbb{R}^n \to \mathbb{R}^n$, whose components $f_i$ for a system of $n$ elements are given as:
$$f_i(\mathbf{p}) = p_i \cdot \frac{n-1+c_i(1-p_i)}{n-1+L_c}, \quad \text{where } L_c = \sum_{k=1}^n c_k p_k (1-p_k).$$
The elements of the Jacobian matrix $J(\mathbf{p})$, defined as $J_{ik} = \frac{\partial f_i}{\partial p_k}$, have the following form:
$$J_{ik}(\mathbf{p}) =
\begin{cases}
    \frac{n-1+c_i(1-p_i)}{n-1+L_c} + p_i \left( \frac{-c_i(n-1+L_c) - c_i(1-2p_i)(n-1+c_i(1-p_i))}{(n-1+L_c)^2} \right), & \text{if } i=k \\
    \\
    -p_i \frac{(n-1+c_i(1-p_i)) \cdot c_k(1-2p_k)}{(n-1+L_c)^2}, & \text{if } i \neq k
\end{cases}
$$

A direct analytical proof of local asymptotic stability for the general case of $n \ge 3$ is infeasible. The stability of the system is determined by the eigenvalues $\lambda$ of the Jacobian matrix $J(\mathbf{p}^\infty)$, which are the roots of the characteristic equation $\det(J(\mathbf{p}^\infty) - \lambda I) = 0$. The problem is that the elements of the matrix $J(\mathbf{p}^\infty)$ depend on the coordinates of the fixed point $p_i^\infty$ and the value $\Lambda$. These quantities, in turn, are highly complex nonlinear functions of the initial system parameters~$c_i$, $i = 1, \ldots, n$. As a result, the coefficients of the characteristic polynomial become exceedingly cumbersome rational expressions of~$c_i$, $i = 1, \ldots, n$.
The characteristic equation is a polynomial of degree $n$. According to the Abel-Ruffini theorem, for polynomials of degree $n \ge 5$, no general algebraic solution exists that expresses the roots in terms of the coefficients using a finite number of arithmetic operations and root extractions.

Given the complexity of an analytical proof, numerical methods are employed to verify stability in each specific case. To automate this process, a dedicated software tool was developed in Python (see~\ref{sec:code}, \emph{stability\_analysis.py}) that performs a stability analysis for any given point ${\bf p}^\infty$ and set of parameters $c_i$. This application supports arbitrary dimensions $n \geq 2$ and any admissible set of constants $c_i$, $i = 1, \ldots, n$. The procedure consists of the following steps:

\begin{enumerate}
    \item For a given point ${\bf p}^\infty$, the set of indices corresponding to non-zero (active) components and the set of indices corresponding to zero components are determined. This division is key, as stability is analyzed separately for perturbations along the active directions (tangential stability) and perpendicular to them (transversal stability).

    \item For the subspace defined by the active components, the corresponding Jacobian sub-matrix $J_{\text{active}}({\bf p}^\infty)$ is formed. Its eigenvalues $\lambda$ are calculated, and the spectral radius $\rho(J_{\text{active}}) = \max |\lambda_i|$ is determined. Stability against tangential perturbations is determined by the criterion $\rho(J_{\text{active}}) < 1$.

    \item For each zero component $p_i^\infty=0$, the corresponding ``normal'' eigenvalue $\lambda_{\text{norm}, i}$ is calculated. If at least one of these values exceeds 1 ($\lambda_{\text{norm}, i} > 1$), the point is considered unstable to perturbations that move it out of the subspace in which it lies.

    \item The final conclusion about the stability of the point ${\bf p}^\infty$ is based on the results of both analyses. A point is stable only if it is stable to both tangential and transversal perturbations. The presence of transversal instability is a sufficient condition for the overall instability of the point, regardless of the result of the tangential stability analysis.
\end{enumerate}

Extensive numerical experiments confirm that for a vector ${\bf p}^{t=0}$ with positive coordinates and randomly generated $c_i$, $i = 1, \ldots, n$ satisfying the conditions of Theorems~\ref{thpop2} and \ref{thpop3}, the spectral radius remains below unity, thus supporting the local asymptotic stability of the computed fixed point.
In cases where one or more $c_i$ cross the critical threshold~$\Lambda(M^*)$, the program consistently detects loss of stability through the emergence of an eigenvalue with $|\lambda_i|>1$.

These consistent analytical considerations and numerical findings allow us to formulate the following general proposition regarding the stability of the system for $n \geq 3$.

Based on these consistent analytical considerations and extensive numerical simulations, we observe a strong pattern regarding the stability of the system for $n \ge 3$. Specifically, for the dynamic conflict system \eqref{eq:1} with an arbitrary set of positive parameters~$c_i$ and initial conditions $p_i^{t=0} > 0$ for all $i=1, \ldots, n$, the trajectory consistently converges to a unique fixed point ${\bf p}^\infty$, whose existence and analytical form are determined by Theorems~\ref{thpop2} and \ref{thpop3}.

The numerical evidence compellingly suggests that this equilibrium ${\bf p}^\infty$ is locally asymptotically stable, while any boundary equilibria (where at least one coordinate {$p_i^\infty = 0$}) are unstable. Furthermore, these empirical observations lead us to the hypothesis that this unique fixed point acts as a global attractor for the set of all initial vectors with strictly positive coordinates.

\subsection{Bifurcation analysis of equilibria}
The conditions derived in Theorem \ref{thpop3}, which determine whether a component $p_i^\infty$ is zero or positive, describe a series of bifurcations (see \cite{GH13, Timokha2025, Avrutin2019, Baiardi2020, Sushko2016}).
It is worth noting that such structural transitions are a fundamental feature of collective dynamics in complex networks, particularly in mean--field models of coupled oscillators with mixed attractive-repulsive (conformist and contrarian) interactions \cite{Pikovsky2001, Daido1992, Hong2011_PRL, Hong2011_PRE, Burylko2020, Burylko2021}.
The specific bifurcation occurring here is a \emph{transcritical bifurcation}, where a stable interior fixed point collides with an unstable boundary fixed point and they exchange stability. This mechanism, common in ecological models \cite{Murray2002}, precisely describes how a component is eliminated from or ``switched on'' in the system.

\subsubsection{Analysis for a single control parameter}
 The instability of boundary equilibria, established in the preceding analysis, is dynamically governed by bifurcations.
 The conditions under which a component of the equilibrium, $p_{i}^{\infty}$, becomes zero or positive essentially describe bifurcations.
 When the parameter $c_i$ crosses a certain threshold value, a qualitative change occurs in the system's asymptotic state: the coordinate $p_{i}^{\infty}$ transitions from a zero value to a positive value, or vice versa. Mathematically, this  means that the system's fixed point moves from a face of the simplex (where $p_i=0$) into the interior of a higher-dimensional face (where $p_i>0$), or vice versa.
 Such phenomena on the boundary of the phase space are interpreted as transcritical bifurcations.

\begin{theorem}
\label{thm:bifurcation}
Let the parameters $\{c_2, \ldots, c_n\}$, $n>2$ be fixed and satisfy the stability condition on the subspace:
$$ c_k > \frac{n-2}{\sum\limits_{j=2, j \neq k}^n \frac{1}{c_j}} \quad \forall k = 2, \ldots, n $$
Let $c_1$ be a bifurcation parameter. Then there exists a critical value $c_1^{\text{crit}}$, defined as:
$$ c_1^{\text{crit}} = \frac{n-2}{\sum\limits_{j=2}^n \frac{1}{c_j}} $$
such that the dynamical system exhibits the following behavior:
\begin{enumerate}
    \item If $c_1 > c_1^{\text{crit}}$, the system has a unique global attractor $\mathbf{p}^\infty$, in which all components are strictly positive ($p_i^\infty > 0$, $\forall \ i=1, \ldots, n$).

    \item If $c_1 < c_1^{\text{crit}}$, the unique attractor of the system is a boundary fixed point $\tilde{\mathbf{p}}^\infty$, for which the first component is zero ($\tilde{p}_1^\infty = 0$) while all others are positive ($\tilde{p}_k^\infty > 0$ for $k \ge 2$).

    \item At $c_1 = c_1^{\text{crit}}$, a transcritical bifurcation occurs, at which point the attractor with positive components $\mathbf{p}^\infty$ merges with the boundary attractor $\tilde{\mathbf{p}}^\infty$, whose first component is zero.
\end{enumerate}
\end{theorem}

\begin{figure}
\centering
\includegraphics[width=0.85\textwidth]{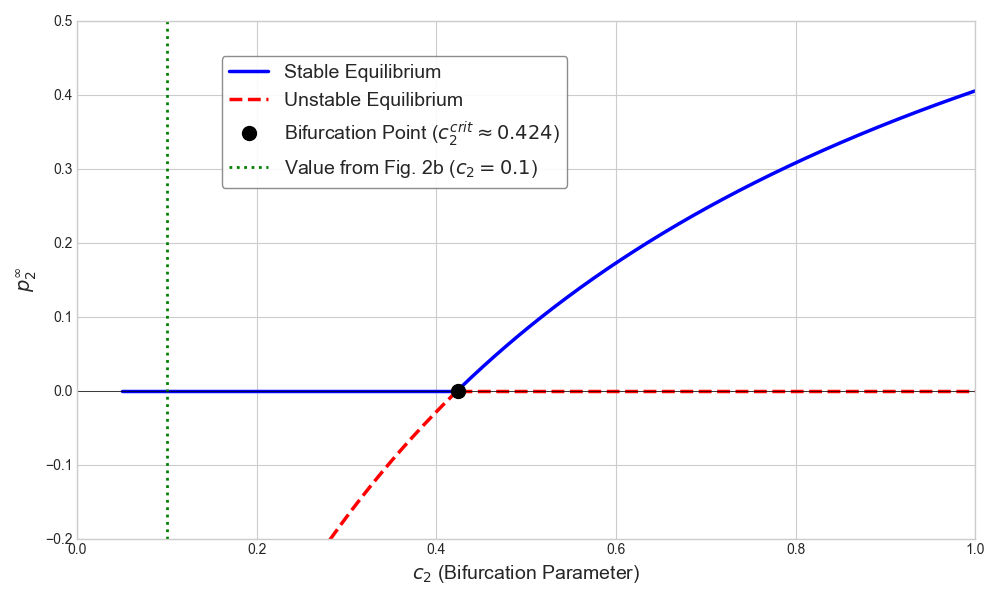}
\caption{Diagram of the transcritical bifurcation for the component $p_2^\infty$ as a function of the parameter $c_2$. This diagram visualizes the conclusions of Theorem~\ref{thm:bifurcation} for parameters from Fig.~\ref{ris:3pop_b}, $c_1=0.8$, $c_3=0.9$.
To provide a complete mathematical characterization, the diagram includes the unstable equilibrium branch in the non-physical region ($p_2^\infty < 0$, red dashed curve).
The vertical green dotted line at $c_2=0.1$ indicates the parameter value from the example in Fig.~\ref{ris:3pop_b}. It shows how the coordinate $p_2^\infty$ changes as the parameter $c_2$ is varied. For $c_2 > c_2^{\text{crit}}$, an exchange of stability occurs: the trivial state $p_2^\infty=0$ loses stability, and the new state with $p_2^\infty > 0$ becomes stable.
The intersection of these two branches and the subsequent exchange of their stability properties are definitive characteristics of a transcritical bifurcation.
}
\label{fig:bifurcation_diagram}
\end{figure}

\begin{proof}
\emph{1. Pre-bifurcation state: $c_1 > c_1^{\text{crit}}$ (Interior fixed point).}
Under this condition, the set of active indices at the steady state is $\mathcal{M}^*=\{1, \ldots, n\}$. The threshold value $\Lambda$ and the first component of the vector $\mathbf{p}^\infty$ depend on the parameter $c_1$:
$$ \Lambda(c_1) = \frac{n-1}{\frac{1}{c_1} + \sum\limits_{j=2}^n \frac{1}{c_j}}, \qquad p_1^\infty(c_1) = 1 - \frac{\Lambda(c_1)}{c_1} $$
The condition for the existence of an interior fixed point ($p_1^\infty(c_1) > 0$) is equivalent to the inequality $c_1 > \Lambda(c_1)$. We expand it as follows:
$$ c_1 > \frac{n-1}{\frac{1}{c_1} + \sum\limits_{j=2}^n \frac{1}{c_j}}, \quad c_1 \left(\frac{1}{c_1} + \sum_{j=2}^n \frac{1}{c_j}\right) > n-1, \quad
1 + c_1 \sum_{j=2}^n \frac{1}{c_j} > n-1$$
$$c_1 \sum_{j=2}^n \frac{1}{c_j} > n-2, \quad c_1 > \frac{n-2}{\sum\limits_{j=2}^n \frac{1}{c_j}}.$$
This proves that as long as $c_1 > c_1^{\text{crit}}$, a stable fixed point exists in the interior of the simplex.

\emph{2. Bifurcation point: $c_1 = c_1^{\text{crit}}$.}
The bifurcation occurs when the interior fixed point reaches the boundary of the simplex, i.e., when $p_1^\infty(c_1) = 0$. This condition is equivalent to the equality $c_1 = \Lambda(c_1)$. As shown above, this equality is achieved precisely at:
$$ c_1 = c_1^{\text{crit}} := \frac{n-2}{\sum\limits_{j=2}^n \frac{1}{c_j}} $$
At this point, the interior attractor $\mathbf{p}^\infty$ and the boundary fixed point $\tilde{\mathbf{p}}^\infty$ coincide, which is the hallmark of a transcritical bifurcation.

\emph{3. Post-bifurcation state: $c_1 < c_1^{\text{crit}}$ (Boundary fixed point).}
When $c_1 < c_1^{\text{crit}}$, the condition $c_1 > \Lambda(c_1)$ is violated, and a fixed point with $p_1^\infty > 0$ is no longer physically possible. The system evolves to a state where the set of active indices reduces to $\mathcal{M}_{\text{new}}^* = \{2, \ldots, n\}$, with cardinality $\gamma(\mathcal{M}_{\text{new}}^*) = n-1$. The new value $\tilde{\Lambda}$ for this subsystem on the face $p_1=0$ is calculated as:
$$ \tilde{\Lambda} = \frac{\gamma(\mathcal{M}_{\text{new}}^*)-1}{\sum\limits_{k \in \mathcal{M}_{\text{new}}^*} \frac{1}{c_k}} = \frac{(n-1)-1}{\sum\limits_{j=2}^n \frac{1}{c_j}} = \frac{n-2}{\sum\limits_{j=2}^n \frac{1}{c_j}} $$
Thus, we see that $\tilde{\Lambda} = c_1^{\text{crit}}$. The new fixed point $\tilde{\mathbf{p}}^\infty$ has the coordinates:
\begin{equation*}
  \tilde{p}_1^\infty = 0,
   \tilde{p}_k^\infty = 1 - \frac{\tilde{\Lambda}}{c_k} = 1 - \frac{c_1^{\text{crit}}}{c_k}, k = 2, \ldots, n.
\end{equation*}
According to the initial theorem condition,
$$c_k > \frac{n-2}{\sum\limits_{j=2, j \neq k}^n \frac{1}{c_j}} > \frac{n-2}{\sum\limits_{j=2}^n \frac{1}{c_j}} = c_1^{\text{crit}}, k \ge 2,$$
which guarantees the positivity of the coordinates $\tilde{p}_k^\infty > 0$. This point is the attractor of the system.

Thus, as the parameter $c_1$ passes through the critical value $c_1^{\text{crit}}$, a qualitative restructuring of the phase portrait occurs: the attractor moves from the interior of the simplex to its boundary, which corresponds to the definition of a transcritical bifurcation.
\end{proof}

\subsubsection{Analysis for two control parameters}
Having established the system's dynamics for a single bifurcation parameter, we now advance the analysis to the scenario where two control parameters, $c_1$ and $c_2$, vary simultaneously.

\begin{theorem}\label{twobif}
Let the parameters $\{c_3, \ldots, c_n\}$, $n>3$, be fixed such that
$$c_k > \frac{n-3}{\sum\limits_{j=3}^n \frac{1}{c_j}} \quad \forall k = 3, \ldots, n$$

 Let $c_1, c_2 > 0$ be control parameters.
Then for for any initial vector ${\bf p}^{t=0} = \{p_i^{t=0}\}_{i=1}^n$, $p_i^{t=0} > 0$ $\forall \ i$, the trajectory ${\bf p}^t$ converges to a unique fixed point ${\bf p}^\infty$, whose structure is determined by the following cases:

 If
 \begin{equation}\label{bif2_1}
   c_1 > \frac{n-2}{\frac{1}{c_2} + \sum\limits_{j=3}^n \frac{1}{c_j}}, \quad c_2 > \frac{n-2}{\frac{1}{c_1} + \sum\limits_{j=3}^n \frac{1}{c_j}},
    \end{equation}
then $p_i^\infty > 0$ for all $i \in \{1, \ldots, n\}$.

If
\begin{equation}\label{bif2_2}
c_1 < \frac{n-2}{\frac{1}{c_2} + \sum\limits_{j=3}^n \frac{1}{c_j}}, \quad c_2 > \frac{n-3}{\sum\limits_{j=3}^n \frac{1}{c_j}},
\end{equation}
then $p_1^\infty = 0$ and $p_i^\infty > 0$ for all $i \in \{2, \ldots, n\}$.

If
 \begin{equation*}
c_2 < \frac{n-2}{\frac{1}{c_1} + \sum\limits_{j=3}^n \frac{1}{c_j}}, \quad c_1 > \frac{n-3}{\sum\limits_{j=3}^n \frac{1}{c_j}},
 \end{equation*}
then $p_2^\infty = 0$ and $p_i^\infty > 0$ for all $i \in \{1, 3, \ldots, n\}$.

If
\begin{equation}\label{bif2_4}
c_1 \le \frac{n-3}{\sum\limits_{j=3}^n \frac{1}{c_j}}, \quad c_2 \le \frac{n-3}{\sum\limits_{j=3}^n \frac{1}{c_j}},
\end{equation}
 then $p_1^\infty = p_2^\infty = 0$ and $p_i^\infty > 0$ for all $i \in \{3, \ldots, n\}$.
\end{theorem}

\begin{figure}[h!]
    \centering
    \includegraphics[width=0.7\textwidth]{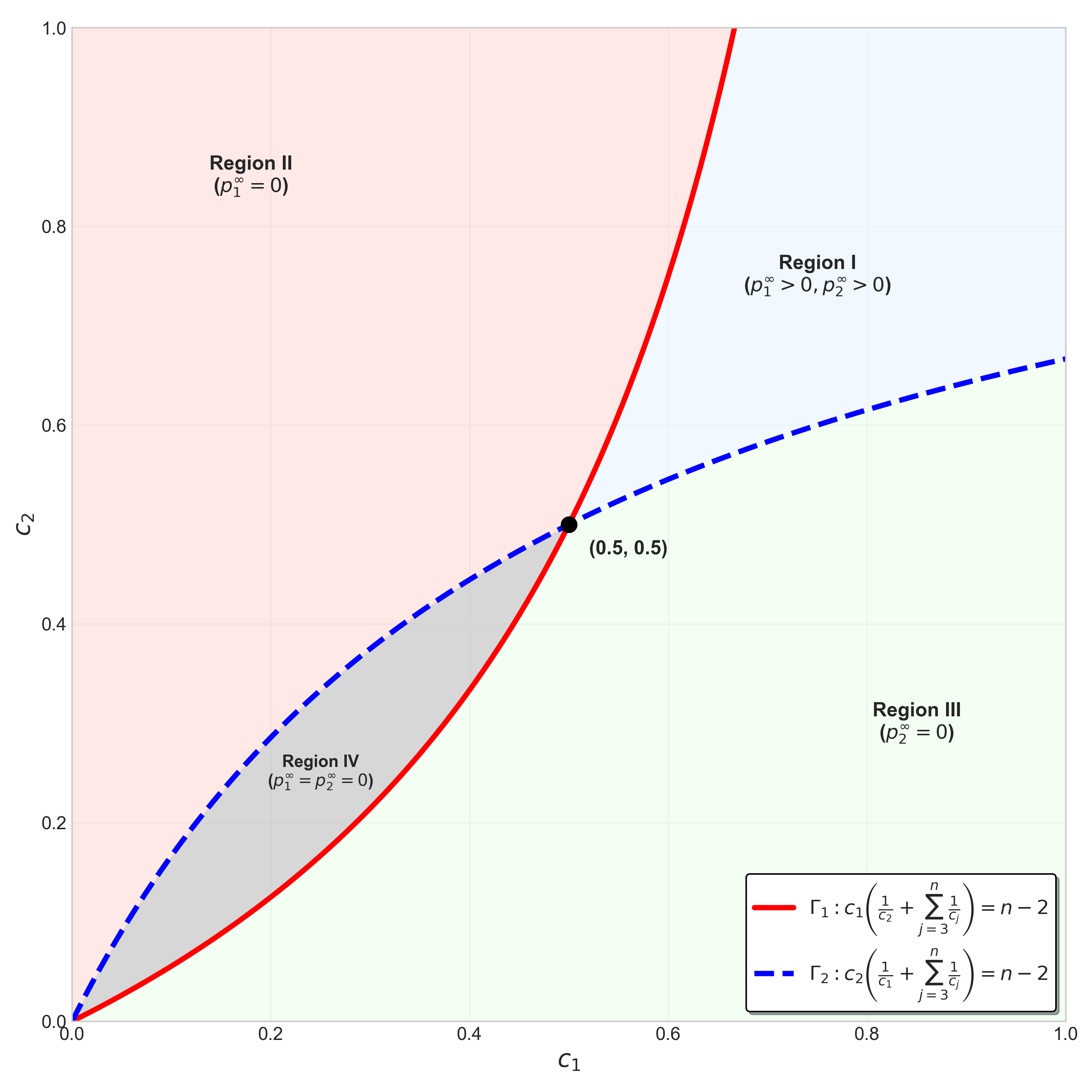}
    \caption{
    Two-parameter bifurcation diagram in the $(c_1, c_2)$ plane for $n=4$ and $c_3 = c_4=1$. The manifolds of transcritical
    bifurcations are defined by the curves
    $\Gamma_1: c_1\left(\frac{1}{c_2}+ \sum_{j=3}^n \frac{1}{c_j}\right) = n-2$
    and $\Gamma_2: c_2\left(\frac{1}{c_1} + \sum_{j=3}^n \frac{1}{c_j}\right) = n-2$.
    Crossing $\Gamma_1$ into Region~II results in $p_1^\infty = 0$, while crossing $\Gamma_2$ into Region~III results in $p_2^\infty = 0$.
    These manifolds partition the parameter space into four regions corresponding to distinct stable equilibria: \emph{Region I}
    (interior fixed point, $p_1^\infty > 0, p_2^\infty > 0$); \emph{Region II}
    (boundary fixed point, $p_1^\infty = 0, p_2^\infty > 0$); \emph{Region III}
    (boundary fixed point, $p_1^\infty > 0, p_2^\infty = 0$); and
    \emph{Region IV} (bounded by both curves, $p_1^\infty = p_2^\infty = 0$).
    The curves $\Gamma_1$ and $\Gamma_2$ intersect transversally at the point $(0.5, 0.5)$,
    where both transcritical bifurcations occur simultaneously.
}
    \label{fig:theorem7_bifurcation}
\end{figure}

\begin{proof}
The proof is based on the result from Theorem~\ref{thpop3}, which states that for a fixed point ${\bf p}^\infty$, its $k$-th coordinate $p_k^\infty$ is positive if and only if the condition $c_k > \Lambda(M^*)$ holds, where $M^* = \{i | p_i^\infty > 0\}$ is the set of indices of active (non-zero) coordinates. The threshold value $\Lambda(M^*)$ is defined as:
$$\Lambda(M^*) = \frac{\gamma(M^*)-1}{\sum\limits_{j \in M^*} \frac{1}{c_j}}$$
where $\gamma(M^*)$ is the cardinality of the set $M^*$.

1. The set of active indices is $M^* = \{1, 2, \ldots, n\}$, thus $\gamma(M^*)=n$. The threshold value is:
$$\Lambda(M^*) = \frac{n-1}{\frac{1}{c_1} + \frac{1}{c_2} + \sum\limits_{j=3}^n \frac{1}{c_j}}$$
The condition $p_i^\infty > 0$ for all $i$ requires the inequalities $c_i > \Lambda(M^*)$ to hold for all $i=1, \ldots, n$. Let's consider the condition for $c_1$:
$$c_1 > \frac{n-1}{\frac{1}{c_1} + \frac{1}{c_2} + \sum\limits_{j=3}^n \frac{1}{c_j}}, \qquad c_1\left(\frac{1}{c_1} + \frac{1}{c_2} + \sum\limits_{j=3}^n \frac{1}{c_j}\right) > n-1,$$
$$1 + \frac{c_1}{c_2} + c_1\sum\limits_{j=3}^n \frac{1}{c_j} > n-1, \qquad c_1\left(\frac{1}{c_2} + \sum\limits_{j=3}^n \frac{1}{c_j}\right) > n-2,$$
$$c_1 > \frac{n-2}{\frac{1}{c_2} + \sum\limits_{j=3}^n \frac{1}{c_j}}.$$
Similarly for $c_2$, we obtain $c_2 > \frac{n-2}{\frac{1}{c_1} + \sum\limits_{j=3}^n \frac{1}{c_j}}$. This corresponds exactly to the conditions~\eqref{bif2_1}.

2. The set of active indices is $M^*_1 = \{2, 3, \ldots, n\}$, thus $\gamma(M^*_1)=n-1$. The threshold value for this subspace is:
$$\Lambda(M^*_1) = \frac{(n-1)-1}{\frac{1}{c_2} + \sum\limits_{j=3}^n \frac{1}{c_j}} = \frac{n-2}{\frac{1}{c_2} + \sum\limits_{j=3}^n \frac{1}{c_j}}$$
The condition $p_1^\infty = 0$ is equivalent to the inequality $c_1 \le \Lambda(M^*_1)$, which yields the first condition of~\eqref{bif2_2}. The condition $p_i^\infty > 0$ for all $i \ge 2$ requires that $c_i > \Lambda(M^*_1)$ for all $i \ge 2$. Let's consider this for $c_2$:
$$c_2 > \frac{n-2}{\frac{1}{c_2} + \sum\limits_{j=3}^n \frac{1}{c_j}}, \qquad c_2\left(\frac{1}{c_2} + \sum\limits_{j=3}^n \frac{1}{c_j}\right) > n-2,$$
$$ 1 + c_2 \sum\limits_{j=3}^n \frac{1}{c_j} > n-2, \qquad c_2 > \frac{n-3}{\sum\limits_{j=3}^n \frac{1}{c_j}}.$$
This is the second condition of~\eqref{bif2_2}.

3. The proof is entirely symmetric to the proof of case~2, by swapping the indices $1 \leftrightarrow 2$.

4. The set of active indices is $M^*_{12} = \{3, \ldots, n\}$, thus $\gamma(M^*_{12})=n-2$. The threshold value is:
$$\Lambda(M^*_{12}) = \frac{(n-2)-1}{\sum\limits_{j=3}^n \frac{1}{c_j}} = \frac{n-3}{\sum\limits_{j=3}^n \frac{1}{c_j}}$$
The conditions $p_1^\infty = 0$ and $p_2^\infty = 0$ are equivalent to the inequalities $c_1 \le \Lambda(M^*_{12})$ and $c_2 \le \Lambda(M^*_{12})$, which correspond exactly to the conditions~\eqref{bif2_4}. The condition $p_i^\infty > 0$ for $i \ge 3$ requires $c_i > \Lambda(M^*_{12})$, which is the initial assumption of the theorem.
\end{proof}

\begin{remark}
The boundaries separating these parameter regions, $\Gamma_1: c_1 \bigg(\frac{1}{c_2} + \sum\limits_{j=3}^n \frac{1}{c_j}\bigg) = n-2$ and $\Gamma_2: c_2 \bigg(\frac{1}{c_1} + \sum\limits_{j=3}^n \frac{1}{c_j}\bigg) = n-2$, are the locations of transcritical bifurcations. The unique fixed point ${\bf p}^\infty$ is the attractor of the system in each respective region.
\end{remark}

\subsubsection{Analysis for multiple control parameters}
The preceding subsections considered cases involving one and two control parameters. We now generalize this analysis to the case of $k$ control parameters, where $1 \le k < n$.
This allows us to describe the conditions under which the system's attractor shifts from a state with $n$ positive components (from the interior of the simplex) to a state on its $k$-dimensional face, where $k$ chosen components become zero.

\begin{theorem}\label{thmkbif}
Let the parameters $\{c_{k+1}, \ldots, c_n\}$, $n > k+1$, be fixed and satisfy the stability condition on their $(n-k)$-component subspace:
$$c_i > \frac{n-k-1}{\sum\limits_{j=k+1}^n \frac{1}{c_j}} \quad \forall i = k+1, \ldots, n.$$
Let $c_1, \ldots, c_k > 0$ be control (bifurcation) parameters. For any initial vector ${\bf p}^{t=0}$ with $p_i^{t=0} > 0$ for all $i$, the trajectory ${\bf p}^t$ converges to a unique fixed point ${\bf p}^\infty$, whose structure is determined by the following cases:

If for all $i = 1, \ldots, k$ the condition holds:
    $$c_i > \frac{n-1}{\sum\limits_{j=1}^n \frac{1}{c_j}}$$
    then $p_i^\infty > 0$ for all $i = 1, \ldots, n$.

If for all $i = 1, \ldots, k$ the condition holds:
    $$c_i \le \frac{n-k-1}{\sum\limits_{j=k+1}^n \frac{1}{c_j}}$$
    where $M^*_k = \{k+1, \ldots, n\}$, then $p_i^\infty = 0$ for $i = 1, \ldots, k$ and $p_j^\infty > 0$ for $j \in M^*_k$.

If for some non-empty subset of indices $S \subset \{1, \ldots, k\}$ the conditions hold:
    $$c_i \le \frac{n-|S|-1}{\sum\limits_{j \in M^*_S} \frac{1}{c_j}} \quad \forall i \in S$$
    $$c_j > \frac{n-|S|-1}{\sum\limits_{j \in M^*_S} \frac{1}{c_j}} \quad \forall j \in \{1, \ldots, k\} \setminus S$$
    where $M^*_S = \{1, \ldots, n\} \setminus S$, then $p_i^\infty = 0$ for $i \in S$ and $p_j^\infty > 0$ for $j \in M^*_S$.

\end{theorem}

\begin{proof}
The proof is a direct generalization of the proof of Theorem~\ref{twobif} and is based on the iterative application of the self-consistent conditions from Theorem~\ref{thpop3}. Each of the described cases corresponds to an attractor in the respective region of the $k$-dimensional parameter space $(c_1, \ldots, c_k)$.

The boundaries between these regions are hypersurfaces on which transcritical bifurcations occur, analogous to the curves $\Gamma_1, \Gamma_2$ depicted in Fig. \ref{fig:theorem7_bifurcation}. The intersection of these hypersurfaces forms a multi-parameter bifurcation structure, which partitions the parameter phase space into regions with different numbers of active components $p_i^\infty > 0$ at equilibrium.
\end{proof}

It is fundamental to note that a bifurcation leading to the simultaneous vanishing of all $n$ coordinates (i.e., $p_i^\infty = 0$ for all $i$) is impossible within this dynamical system. The reason lies in the model's core mathematical construction: the state vector ${\bf p}^t$ is a stochastic vector, which must satisfy the normalization constraint $\sum_{i=1}^{n} p_i^t = 1$ for all $t \ge 0$. The vector of  equilibrium ${\bf p}^\infty = (0, 0, \ldots, 0)$ would imply $\sum_{i=1}^{n} p_i^\infty = 0$, which violates this conserved quantity and does not belong to the state space (the standard $(n-1)$--dimensional simplex).

The most extreme topological change the system can undergo is a transition where $n-1$ coordinates vanish simultaneously. This corresponds to the system's attractor locating at one of the vertices of the simplex, ${\bf p}^\infty = {\bf e}_k = (0, \ldots, 1, \ldots, 0)$. In applied terms, this state represents the total aggregation of the resource into a single component, which is the maximal possible aggregation that still perfectly satisfies the normalization constraint.

\subsection{Criteria for optimizing extremal coordinates}
The analytical formulas for the equilibrium vector ${\bf p}^\infty$  naturally lead to questions of control and optimization. It is important to note that the dynamic conflict system \eqref{eq:1} is not a gradient system; it does not explicitly minimize a potential function. However, the fact that the asymptotic state is uniquely and analytically determined by the parameters $c_i$ allows us to treat these parameters as control variables.

The following propositions explore these criteria, which can be viewed as solutions to an optimal control problem in a deterministic mean--field system, a topic relevant to modern resource allocation challenges \cite{Lasry2007, Hung2006}.

Without loss of generality, let's assume that for all $k=1, \ldots, n$, condition~\ref{eneq_all} holds, meaning all coordinates $p_k^\infty$ are strictly positive and $\gamma(M^*)=n$. Then
\begin{equation*}
p_k^\infty = 1 - \frac{\Lambda}{c_k}, \quad \text{where} \quad \Lambda = \frac{n-1}{\sum\limits_{m=1}^n \frac{1}{c_m}}.
\end{equation*}
This implies that all parameters $c_k$ satisfy the condition $c_k > \Lambda$, $k=1, \ldots, n$.

\begin{figure}[h]
\center{\includegraphics[scale=0.6]{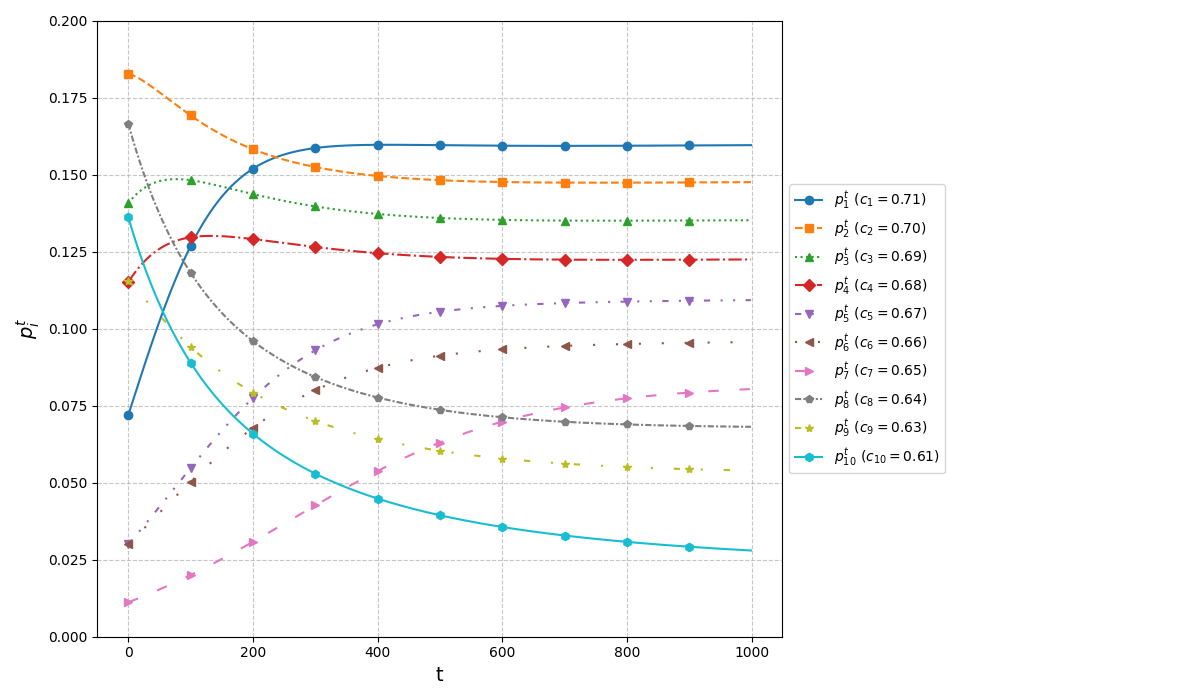}}
 \caption{Evolution of the vector coordinates from the initial state
${\bf p}^0 = (0.072; 0.183; 0.141; 0.115; 0.030;  0.030;
 0.011; 0.1666; 0.116; 0.136)$ to the equilibrium
${\bf p}^\infty = (0.160; 0.148; 0.135; 0.122; 0.109; 0.096;
 0.080; 0.068; 0.054; 0.028)$
under the action of conflict transformation~\eqref{eq:2}. All parameters $c_k$ satisfy the condition $c_k > \Lambda$, $k=1, \ldots, 10$.
The equilibrium is entirely determined by the set of constants $\{c_i\}_{i=1}^{10}$,
illustrating the optimization principles from~\ref{thm:max_coord}--\ref{thm:simultaneous_opt}:
the coordinate $p_1^t$ with the maximum $c_1=0.71$ converges to the highest value,
while $p_{10}^t$ with the minimum $c_{10}=0.61$ converges to the lowest.}
\label{fig-3}
\end{figure}

\begin{proposition}[On maximizing a coordinate of the equilibrium] \label{thm:max_coord}
For maximizing the coordinate $p_i^\infty$ for some fixed index $i$, it is necessary and sufficient that:
\begin{description}
    \item[(a)]\label{desc-max:item-a} The parameter $c_i$ is maximal among all parameters $c_k$: $c_i = \max\limits_{k \in \{1,\ldots,n\}} \{c_k\}.$
    \item[(b)]\label{desc-max:item-b} The parameters $c_m$ for $m \ne i$ are as small as possible, provided that $c_m > \Lambda$.
\end{description}
\end{proposition}

\begin{proof}
To maximize the  coordinate $p_i^\infty = 1 - \frac{\Lambda}{c_i}$, it is necessary to minimize the positive fraction $\frac{\Lambda}{c_i}$. This is equivalent to maximizing the denominator $c_i \sum\limits_{m=1}^n \frac{1}{c_m}$. Let's expand this product:
$$c_i \sum_{m=1}^n \frac{1}{c_m} = c_i \left( \frac{1}{c_i} + \sum\limits_{m \ne i} \frac{1}{c_m} \right) = 1 + c_i \sum\limits_{m \ne i} \frac{1}{c_m}$$
The expression $1 + c_i \sum\limits_{m \ne i} \frac{1}{c_m}$ reaches its maximum value when $c_i$ is maximal (condition (a)), and the sum $\sum\limits_{m \ne i} \frac{1}{c_m}$ is maximal. The latter condition holds when the values $c_m$ (for $m \ne i$) are as small as possible, provided that $c_m > \Lambda$ (condition (b)).
Thus, conditions (a) and (b) are necessary and sufficient for maximizing $p_i^\infty$.
\end{proof}

\begin{proposition}[On minimizing a coordinate of the equilibrium] \label{thm:min_coord}
To minimize the coordinate $p_j^\infty$ (given $p_j^\infty > 0$) for some fixed index $j$, it is necessary and sufficient that:
\begin{description}
    \item[(a)] The parameter $c_j$ is minimal among all parameters $c_k$ (satisfying $c_j > \Lambda$).
    \item[(b)] The parameters $c_m$ for $m \ne j$ are as large as possible.
\end{description}
\end{proposition}

\begin{proof}
To minimize the  coordinate $p_j^\infty = 1 - \frac{\Lambda}{c_j}$ (while keeping $p_j^\infty > 0$), it is necessary for the fraction $\frac{\Lambda}{c_j}$ to be maximal, but less than 1. This is equivalent to minimizing the denominator $c_j \sum\limits_{m=1}^n \frac{1}{c_m}$, provided it remains strictly greater than $n-1$.
The expression $1 + c_j \sum\limits_{m \ne j} \frac{1}{c_m}$ will be minimal (given $p_j^\infty > 0$) when $c_j$ is minimal (condition (a)), and the sum $\sum\limits_{m \ne j} \frac{1}{c_m}$ is minimal. The latter condition holds when the values $c_m$ (for $m \ne j$) are as large as possible (condition (b)).
Thus, conditions (a) and (b) are necessary and sufficient for minimizing $p_j^\infty$ (given $p_j^\infty > 0$).
\end{proof}

\begin{proposition}[On simultaneously optimizing the coordinates of the equilibrium]\label{thm:simultaneous_opt}
For the simultaneous maximization of $p_i^\infty$ for some fixed index $i$ and minimization of $p_j^\infty$ for some fixed index $j$, $j \neq i$, it is necessary and sufficient that the following conditions hold:
\begin{description}
    \item[(a)] The parameter $c_i$ is maximal among all parameters $c_k$: $c_i = \max\limits_{k \in \{1,\ldots,n\}} \{c_k\}$.
    \item[(b)] The parameter $c_j$ is minimal among all parameters $c_k$: $c_j = \min\limits_{k \in \{1,\ldots,n\}} \{c_k\}$.
    \item[(c)] The parameters $c_k$ for $k \notin \{i, j\}$ are as small as possible, i.e., close to $c_j$, provided that $c_k > \Lambda$.
\end{description}
\end{proposition}

\begin{proof}
From Theorem \ref{thm:max_coord} it follows that to maximize $p_i^\infty$, $c_i$ must be maximal, and $c_m$ ($m \ne i$) must be minimal (provided $c_m > \Lambda$).
From Theorem \ref{thm:min_coord} it follows that to minimize $p_j^\infty$, $c_j$ must be minimal (provided $c_j > \Lambda$), and $c_m$ ($m \ne j$) must be maximal.

For these conditions to hold simultaneously, conditions (a) and (b) of this theorem are necessary.
Let's consider the influence of parameters $c_k$ (where $k \notin \{i, j\}$) on the value of $\Lambda = \frac{n-1}{\sum\limits_{m=1}^n \frac{1}{c_m}}$.
To maximize $p_i^\infty = 1 - \frac{\Lambda}{c_i}$, it is necessary to minimize $\Lambda$. This is achieved when the sum $\sum\limits_{m=1}^n \frac{1}{c_m}$ is maximal.
To minimize $p_j^\infty = 1 - \frac{\Lambda}{c_j}$ (i.e., $p_j^\infty \to 0^+$), it is necessary that $\frac{\Lambda}{c_j} \to 1^-$, which means $\Lambda \to c_j^+$.

Let
$c_i = \max\limits_{k \in \{1, \ldots, n\}} \{c_k\}$, $c_j = \min\limits_{k \in \{1, \ldots, n\}} \{c_k\}.$
To minimize $\Lambda$ (which is beneficial for $p_i^\infty$), the sum $\sum_{m=1}^n \frac{1}{c_m}$ must be large. Since $c_i$ is large and $c_j$ is small, $1/c_i$ is small and $1/c_j$ is large. To maximize the sum, the terms $1/c_k$ for $k \notin \{i, j\}$ must also be large. This is achieved if $c_k$ for $k \notin \{i, j\}$ are small, i.e., close to $c_j$ (condition (c)).

If $c_k \approx c_j$ for $k \notin \{i, j\}$, then $\sum_{m=1}^n \frac{1}{c_m} \approx \frac{1}{c_i} + \frac{n-1}{c_j}$ (assuming all $c_k, k \ne i$, are approximately equal to $c_j$).
Then $\Lambda \approx \frac{n-1}{\frac{1}{c_i} + \frac{n-1}{c_j}}$.
If $c_i \gg c_j$, then $1/c_i$ is small compared to $(n-1)/c_j$, and then $\Lambda \approx \frac{n-1}{(n-1)/c_j} = c_j$.
This value $\Lambda \approx c_j$ satisfies the condition $p_j^\infty \to 0^+$ (since $\Lambda$ is close to $c_j$ from below, so that $c_j > \Lambda$). Simultaneously, such a minimal value of $\Lambda$ (close to the smallest $c_k$, i.e., $c_j$) together with a maximal $c_i$ leads to the maximization of $p_i^\infty$.

Thus, conditions (a), (b), and (c) are necessary and sufficient for the simultaneous optimization of $p_i^\infty$ and $p_j^\infty$ in the sense described.
\end{proof}
\section{Future outlook: from fixed points to complex dynamics via delayed feedback}

A significant extension of the present model involves transitioning from static favorability parameters $c_i$ to dynamic coefficients that depend on the system's history. This modification reflects real-world scenarios where prolonged dominance of one component can trigger negative feedback mechanisms, such as resource depletion, regulatory intervention, or social fatigue.

A particularly well-studied and powerful mechanism for this is delayed feedback, famously introduced by Pyragas as a method for controlling chaos \cite{Pyragas1992, SS2008}. While originally conceived for control, delayed feedback is also a potent source of complex dynamics. Introducing a delay $\tau$ transforms the system into an infinite-dimensional one, known to undergo bifurcations leading to stable limit cycles and chaos \cite{Er2009, Wei2024, Burylko2023, Burylko2022, Panchuk2021, Avrutin2023}. Such complex mean--field dynamics provide a clear theoretical pathway from our fixed-point dynamics to the rich behaviors observed in macroscopic real--world systems, including biophysically grounded neural populations \cite{Majhi2019, Cakan2020} and complex quantum states \cite{Trigueros2024}.

This dynamic can be modeled by introducing a delayed dependence on the component's own share. The favorability coefficient $c_i$ at time $t$ is then defined as:
\begin{equation*} \label{eq:dynamic_c_revised}
c_i(t) = c_{i, \text{max}} - \beta \cdot p_i(t-\tau)
\end{equation*}
where $c_{i, \text{max}} = \max\limits_i \{c_i\}_{i=1}^n$ is the baseline favorability, $\tau$ is the time lag, and the parameter~$\beta$ represents the strength of the negative feedback.

\begin{figure}[h!]
\begin{minipage}{0.9\linewidth}
\center{\includegraphics[width=1\linewidth]{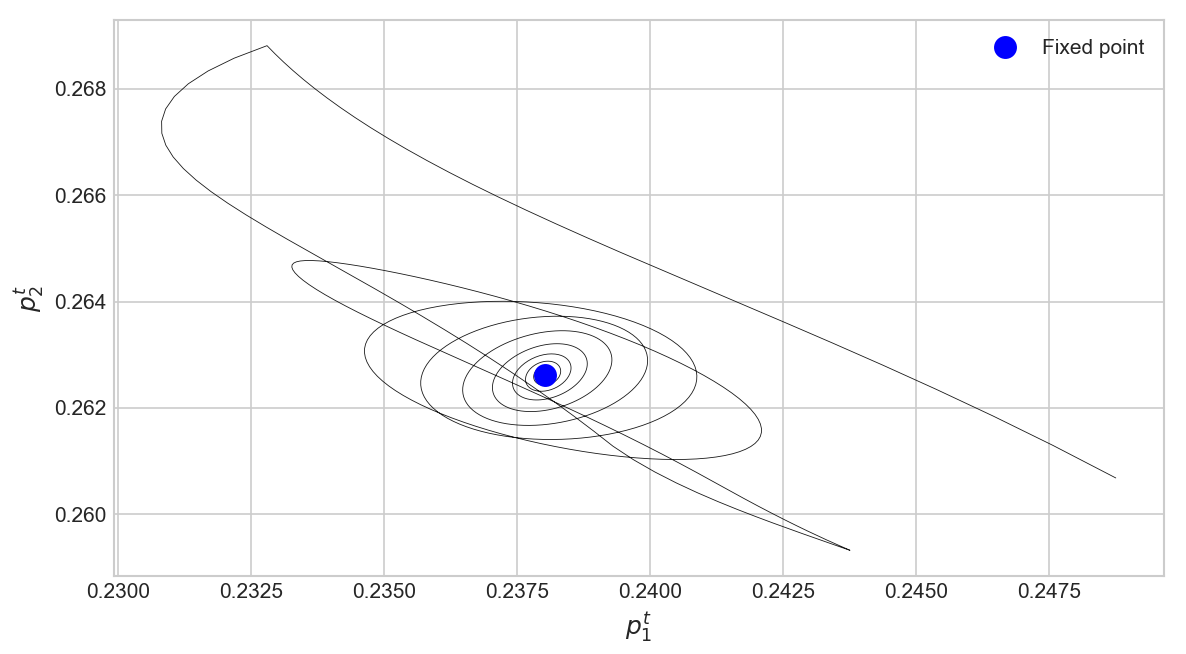}} a) Stable fixed point, $\beta=1.2$ \\
\end{minipage}
\vfill
\begin{minipage}{0.47\linewidth}
\center{\includegraphics[width=1.1\linewidth]{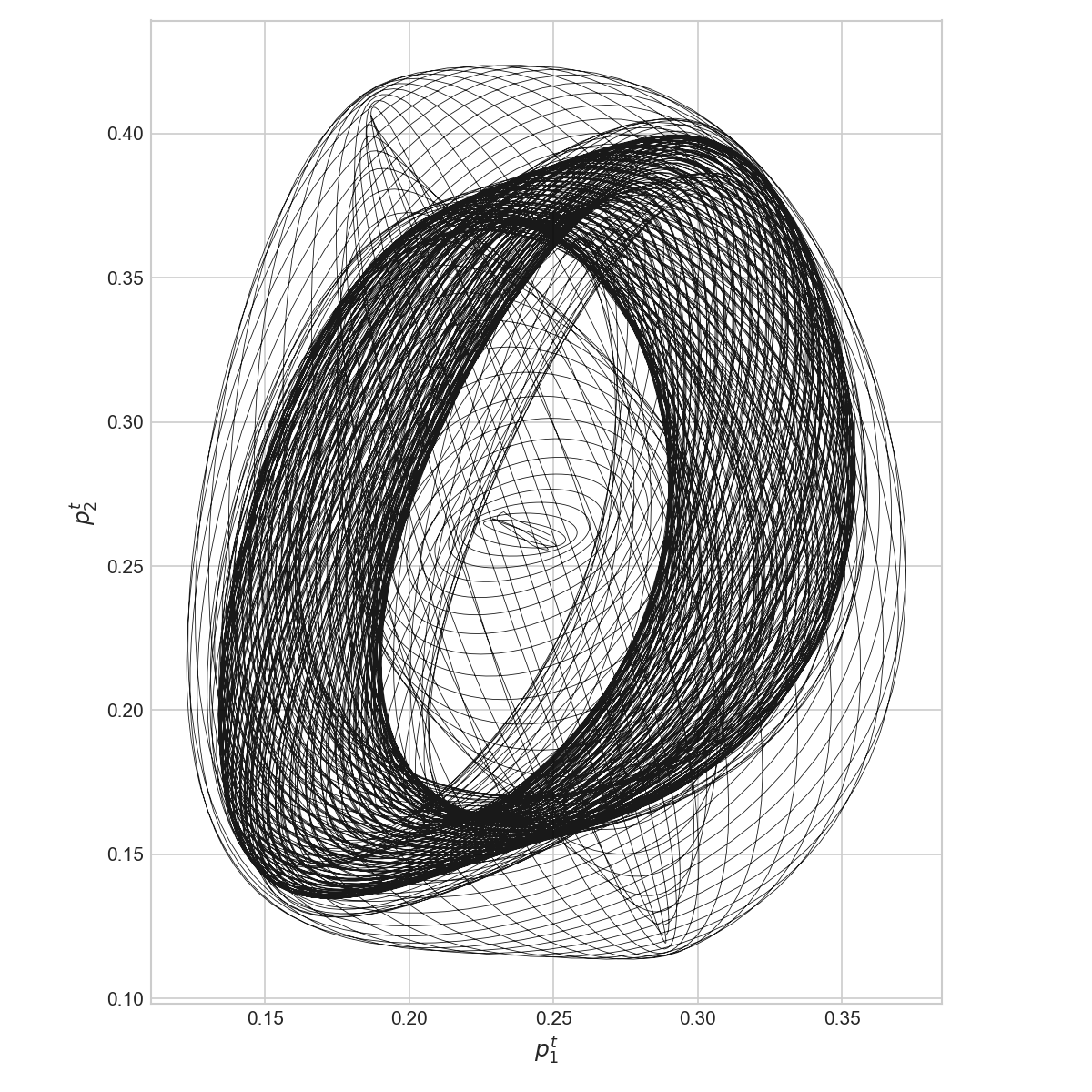}} b) Quasi-periodic orbit, $\beta=1.5$ \\
\end{minipage}
\hfill
\begin{minipage}{0.47\linewidth}
\center{\includegraphics[width=1.12\linewidth]{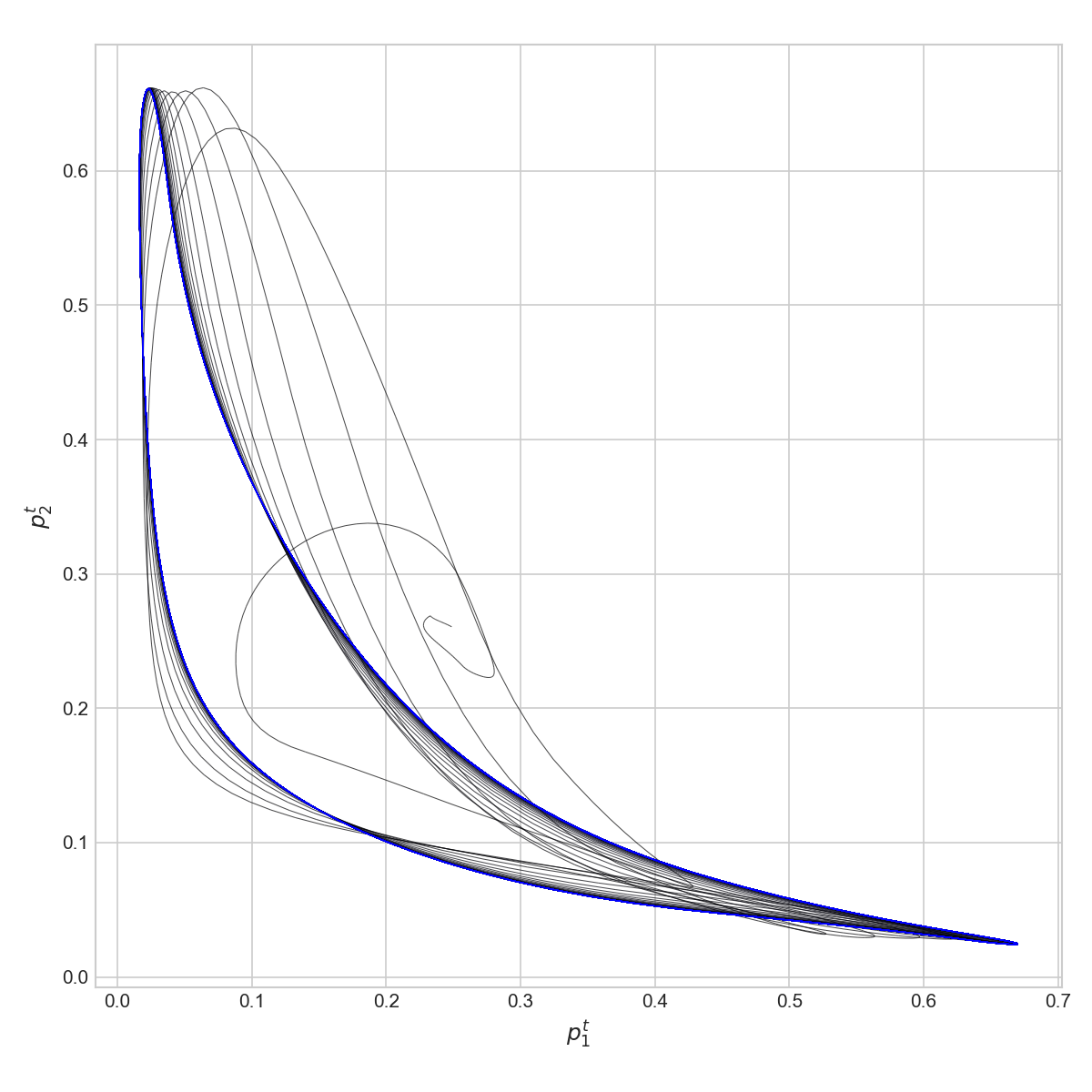}} c) Periodic orbit, $\beta=3$ \\
\end{minipage}
\caption{Phase portrait (attractor in the $p_1^t$, $p_2^t$ space), where ${\bf p}^0 = (0.25, 0.26, 0.24, 0.25)$, ${\bf c}=(0.9, 0.85, 0.95, 0.8)$, $\tau=30$.}
\label{fig:cycle_chaos_plots}
\end{figure}

The introduction of this feedback mechanism reveals a rich spectrum of dynamical behaviors far beyond a single fixed point. Numerical simulations demonstrate that the feedback strength $\beta$ acts as a critical bifurcation parameter, driving the system's complexity and governing the transition from fixed-point dynamics to complex aperiodic behavior. The bifurcation diagram (Fig.~\ref{fig:cycle_chaos_diag}) is constructed by plotting the asymptotic values of the trajectory's first coordinate, $p_1^t$, after discarding initial transient dynamics.

For moderate values of $\beta$, the system exhibits periodic and quasi-periodic oscillations (Fig.~\ref{fig:cycle_chaos_plots}b, c). In terms of the resource competition model, this mathematical regime represents a dynamic equilibrium where no single component achieves permanent dominance, leading instead to a perpetual and predictable periodic or quasi-periodic redistribution of resources.

As $\beta$ increases past approximately 1.5, we phenomenologically observe that the system loses its stable fixed point and transitions into a broad regime of highly complex, aperiodic, and unpredictable oscillations. The trajectory remains confined to a specific region of the phase space, which we provisionally describe as a complex aperiodic (chaotic-like) attractor, suggesting significant sensitivity to initial conditions. These dynamics remain highly complex throughout a wide parametric window $\beta \in [1.5, 4.1]$, where simulations reveal a dense bifurcation structure interspersed with sudden topological changes.

\begin{figure}[h!]
\center{\includegraphics[scale=0.45]{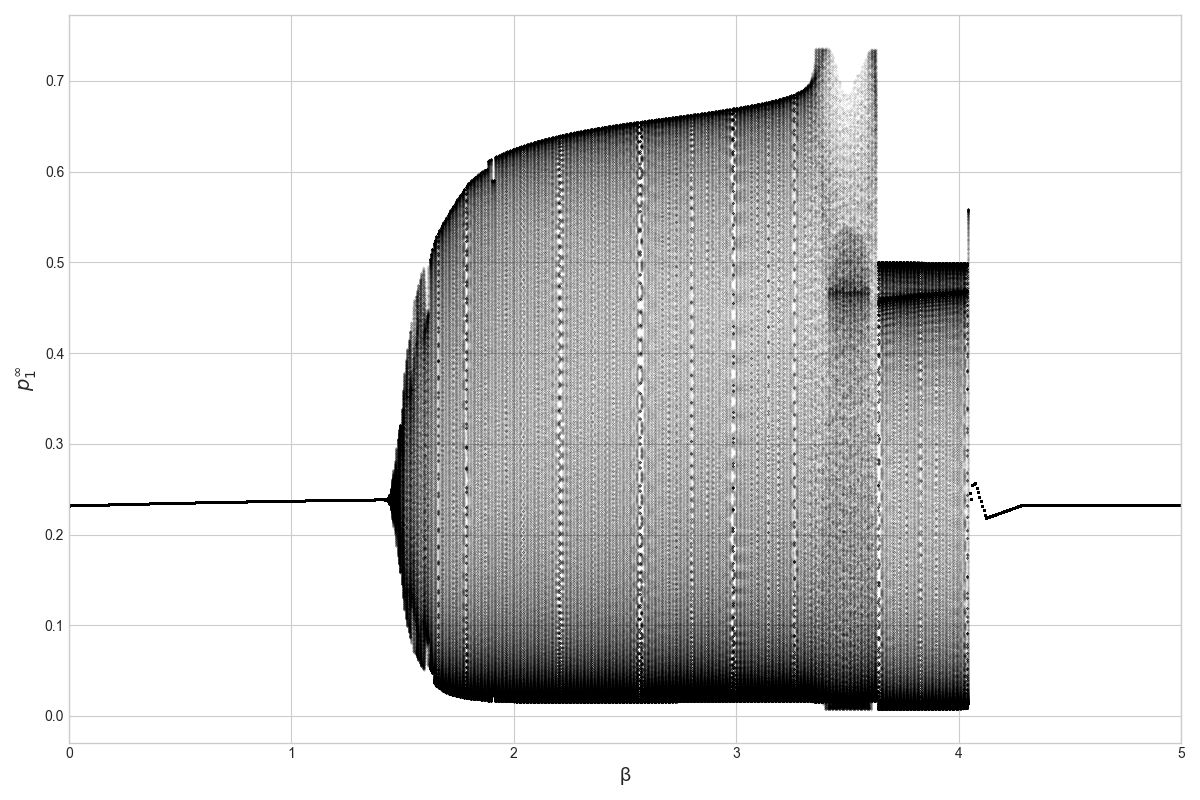}}
\caption{Bifurcation diagram of the extended system with delayed feedback for ${\bf c}=(0.9, 0.85, 0.95, 0.8)$
and $\tau=30$. Asymptotic values of $p_1^t$ (after discarding $10^4$ transient steps) are plotted as a function of $\beta$,
revealing a stable fixed point for $\beta \lessapprox 1.5$, a broad complex regime for $\beta \in [1.5, 4.1]$, and a return
to a stable fixed point following a global crisis near $\beta \approx 4.1$.}
\label{fig:cycle_chaos_diag}
\end{figure}

While a rigorous topological classification of these bifurcations falls outside the scope of this paper, these numerical results clearly confirm that delayed attractive interaction can drive the considered system through highly unpredictable oscillatory regimes and abrupt structural shifts, providing a rich ground for future analytical studies.

This extended framework provides a powerful tool for modeling complex real-world systems that exhibit not only stable equilibria or regular cycles but also irregular and unpredictable fluctuations. A full bifurcation analysis, mapping the precise route to complex aperiodic dynamics as a function of the system parameters $\beta$ and $\tau$, remains a key objective for future work.

\section{Conclusion}
This study provides a comprehensive analysis of a dynamical system that models the redistribution of a conserved
quantity (e.g., resource, capital, or influence) in a discrete state space under the influence of attractive interaction.
The core mathematical findings demonstrate that trajectories of this system always converge to a stable fixed point.
The topological location of this equilibrium~--- whether it lies in the interior of the simplex (all $p_i^\infty > 0$)
or on its boundary (some $p_i^\infty = 0$)~--- is analytically determined by the interaction constants~$c_i$. In terms
of the applied model, this dictates whether the system reaches a state of coexistence or heavily aggregated resource concentration.

When all parameters~$c_i$ are equal, all trajectories with strictly positive initial coordinates converge to a unique
attractor located exactly at the barycenter of the simplex ($p_i^\infty = 1/n$ for all $i$). This result admits
a natural interpretation: it describes a theoretical pathway to achieving perfect equality in resource distribution, which is
important for the analysis of idealized social or economic systems.

The research is further extended by analyzing the influence of heterogeneous constants~$c_i$ on the equilibrium.
The explicit analytical form of the fixed point reveals a strict functional dependence: higher values of a parameter~$c_i$ strictly
maximize the corresponding coordinate $p_i^\infty$. Furthermore, the study establishes exact bifurcation thresholds under which a
specific coordinate asymptotically vanishes ($p_i^\infty = 0$). In the context of resource competition, these mathematical phenomena
represent, respectively, the systematic concentration of the resource and the complete competitive exclusion (displacement) of a component
from the system.

The proposed analytical criteria translate these mathematical results into actionable conditions: by tuning the
parameters~$c_i$, one can deliberately steer the system toward a desired equilibrium configuration, whether maximizing a specific
coordinate or driving others to zero. This offers a concrete theoretical tools
for managing systems exhibiting attractive interactions.

The developed framework~--- comprising explicit equilibrium formulas, bifurcation thresholds, and stability
criteria~--- provides a self--contained mathematical foundation for the analysis of discrete--time attractive interaction models, with
direct applications to resource competition, opinion dynamics, and wealth distribution.

\section*{Acknowledgments}
This work was partially supported by the Simons Foundation grant SFI-PD-Ukraine-00014586 (O.R.S.) and by the project “Mathematical modeling of complex systems and security-related processes” (No. 0125U000299).

\label{sec:code}
\section*{Code Availability}\label{sec:code}
The Python code implementing the algorithms presented in this paper is available in a public GitHub repository: https://github.com/kseniajasko/fixed-point-analysis-conflict-model.

\end{document}